\documentclass[a4paper, 11pt]{amsart} 
\usepackage[
  margin=1.9cm,
  includefoot,
  footskip=30pt,
]{geometry}
\usepackage{amsmath,amssymb,amscd}
\usepackage{amsfonts}
\usepackage[english]{babel}
\usepackage[leqno]{amsmath}
\usepackage{amssymb,amsthm}
\usepackage{mathrsfs} 
\usepackage{amscd}
\usepackage{enumerate}
\usepackage{epsfig}
\usepackage{relsize}
\usepackage{layout}
\usepackage[usenames,dvipsnames]{xcolor}
\usepackage[backref=page]{hyperref}
\usepackage{tikz}
\hypersetup{
 colorlinks,
 citecolor=Green,
 linkcolor=Red,
 urlcolor=Blue}
\usepackage[matrix,arrow,tips,curve]{xy}
\input{xy}
\xyoption{all}
\definecolor{darkgreen}{rgb}{0.0, 0.7, 0.0}
\definecolor{purple}{rgb}{0.5, 0.0, 0.5}
\definecolor{red}{rgb}{0.8, 0.2, 0.0}

\newtheorem{thm}{Theorem}[section]
\newtheorem{bthm}{Theorem}

\newtheorem{lemma}[thm]{Lemma}

\newtheorem{prop}[thm]{Proposition}
\newtheorem{cor}[thm]{Corollary}
\newtheorem{claim}[thm]{Claim}

\numberwithin{equation}{section}
\theoremstyle{definition}
\newtheorem{defi}[thm]{Definition}

\newtheorem{notation}[thm]{Notation}

\theoremstyle{remark}
\newtheorem{remark}[thm]{Remark}

\newtheorem{example}[thm]{Example}
\newcommand{\Z}{\mathbb{Z}}

\newcommand{\R}{\mathbb{R}}

\newcommand{\Pic}{\operatorname{Pic}}

\def \Im{{\rm Im}}

\def \P{\mathbb{P}}

\def \ZZ{\mathbb{Z}}

\def \F{\mathcal F}

\def\I{{\mathcal J}}
\def \L{\mathcal L}
\def \E{\mathcal E}
\def \G{\mathcal G}

\def\O{\mathcal O}

\def\M0{\mathcal M^0}

\newcommand{\rk}{\operatorname{rk}}

\title[On varieties with Ulrich twisted tangent bundles]{On varieties with Ulrich twisted tangent bundles}

\author[A.F. Lopez, D. Raychaudhury]{Angelo Felice Lopez* and Debaditya Raychaudhury**}

\address{\hskip -.43cm Angelo Felice Lopez, Dipartimento di Matematica e Fisica, Universit\`a di Roma
Tre, Largo San Leonardo Murialdo 1, 00146, Roma, Italy. e-mail {\tt lopez@mat.uniroma3.it}}

\address{\hskip -.43cm Debaditya Raychaudhury, Department of Mathematics, University of Arizona, 617 N. Santa Rita Ave., Tucson, Arizona 85721, USA. email: {\tt draychaudhury@math.arizona.edu}}

\thanks{* Research partially supported by PRIN ``Advances in Moduli Theory and Birational Classification'', GNSAGA-INdAM and the MIUR grant Dipartimenti di Eccellenza 2018-2022.}

\thanks{** Research partially supported by a Simons Postdoctoral Fellowship from the Fields Institute for Research in Mathematical Sciences}


\begin{document}

\begin{abstract} 
We study varieties $X \subseteq \P^N$ of dimension $n$ such that $T_X(k)$ is an Ulrich vector bundle for some $k \in \Z$. First we give a sharp bound for $k$ in the case of curves. Then we show that $k \le n+1$ if $2 \le n \le 12$. We classify the pairs $(X,\O_X(1))$ for $k=1$ and we show that, for $n \ge 4$, the case $k=2$ does not occur.
\end{abstract}

\maketitle

\section{Introduction}

Let $X \subseteq \P^N$ be a smooth irreducible variety of dimension $n \ge 1$. As is well known, the study of vector bundles on $X$ can give important geometrical information about $X$ itself. Regarding this, one of the most interesting family of vector bundles associated to $X$ and its embedding, that received a lot of attention lately, is that of Ulrich vector bundles, that is bundles $\E$ such that $H^i(\E(-p))=0$ for all $i \ge 0$ and $1 \le p \le n$. The study of such bundles is closely related with several areas of commutative algebra and algebraic geometry, and often gives interesting consequences on the geometry of $X$ and on the cohomology of sheaves on $X$ (see for example in \cite{es, b1, cmp} and references therein). 

Perhaps the most challenging question in these matters is whether every $X \subseteq \P^N$ carries an Ulrich vector bundle (see for example \cite[page 543]{es}). It comes therefore very natural to ask if usual vector bundles associated to $X$ can be Ulrich. Also, since Ulrich vector bundles are globally generated, it is better to consider twisted versions, by some divisor $D$, of the usual bundles associated to $X$. On the other hand, in order to keep some relation with the embedding and to have a better chance for global generation, we will consider twists by $D=kH$, for some integer $k$. The cases of the (twisted) normal, cotangent, restricted tangent and cotangent bundles have been dealt with in \cite{lo}, with an essentially complete classification. 

In this paper we study the more delicate question: for which integers $k$ one has that $T_X(k)$ is an Ulrich vector bundle?

Ulrich vector bundles have special cohomological features, but also numerical ones. This makes the above question rather tricky. It is easy to show that $k \ge 0$ unless $(X,\O_X(1),k)=(\P^1,\O_{\P^1}(1),-2)$. In the case $k=0$, a recent result \cite[Prop.~4.1, Thm.~4.5]{bmpt} gives a classification: $(X,\O_X(1))=(\P^1,\O_{\P^1}(3)), (\P^2,\O_{\P^2}(2))$ (we will give a new and simple proof in section \ref{sez7}; another proof is given in \cite{c2}). On the other hand, for $k \ge 1$, the question is more subtle as we will see below.

In the case of curves, one sees that $k=1$ is not possible (see Lemma \ref{posi}(i)), while the cases $k=2, 3$ can be dealt with on any curve (see Lemma \ref{k=2curve} and Example \ref{k=3curve}). On the other hand, the following sharp  bound holds, showing that for curves $k$ can be as large as wanted.

\begin{bthm}
\label{gen}

\hskip 3cm

Let $X \subseteq \P^N$ be a smooth irreducible curve of genus $g$. If $T_X(k)$ is an Ulrich line bundle, then
\begin{equation}
\label{bdk}
k \le \frac{\sqrt{8g+1}-1}{2} 
\end{equation}
and equality holds if and only if $k$ is even and either $X$ is one of the curves \eqref{cla} lying on a smooth cubic or $X$ is a curve of type $(\frac{k}{2}+1,k+2)$ on a smooth quadric. Also, in both cases, $T_X(k)$ is an Ulrich line bundle, hence the bound is sharp for every even $k \ge 0$. Moreover, if $X$ has general moduli, then $k \le 4$.
\end{bthm}
As far as we know, only curves show this kind of behavior, meaning that $k$ is not bounded in terms of the dimension (a somewhat bad bound can also be given in terms of the degree, see Lemma \ref{bigbound}). As supporting evidence, we prove the following

\begin{bthm}
\label{bou}

\hskip 3cm

Let $X \subseteq \P^N$ be a smooth irreducible variety of dimension $n$ such that $2 \le n \le 12$. If $T_X(k)$ is an Ulrich vector bundle, then $k \le n+1$.
\end{bthm}

We should point out that, for $n \ge 2$, we know no examples with $k \ge 2$ and only one example with $k=1$. As a matter of fact, the case $k=1$ can be completely characterized, as follows

\begin{bthm}
\label{k=1}

\hskip 3cm

Let $X \subseteq \P^N$ be a smooth irreducible variety of dimension $n \ge 1$. Then $T_X(1)$ is an Ulrich vector bundle if and only if $(X,\O_X(1))=(S_5,-2{K_{S_5}})$, where $S_5$ is a Del Pezzo surface of degree $5$.
\end{bthm}
On the other hand, for $k=2$, we have

\begin{bthm}
\label{k=2}

\hskip 3cm

Let $X \subseteq \P^N$ be a smooth irreducible variety of dimension $n \ge 4$. Then $T_X(2)$ is not an Ulrich vector bundle. 
\end{bthm}
We do not know what happens for $k=2, n=3$, even though some evidence suggests that it might not be possible. Also, for surfaces, the cases $k=2, 3$ point out to the possible existence, that needs to be further investigated, of some minimal surfaces of general type, as shown in Lemma \ref{sup1} and Proposition \ref{sup}.

Finally, in any dimension, another interesting case is the one in which $\omega_X$ and $\O_X(1)$ are numerically proportional. This is dealt with in Theorem \ref{prop}, Corollaries \ref{sottoc} and \ref{altroprop}.

\section{Notation}

Throughout the paper we work over the complex numbers. Moreover we henceforth establish the following

\begin{notation}
\label{nota}

\hskip 3cm

\begin{itemize} 
\item $X$ is a smooth irreducible variety of dimension $n \ge 1$.
\item $H$ is a very ample divisor on $X$. 
\item For any sheaf $\G$ on $X$ we set $\G(l)=\G(lH)$.
\item $d=H^n$ is the degree of $X$.
\item $C$ is a general curve section of $X$ under the embedding given by $H$.
\item $S$ is a general surface section of $X$ under the embedding given by $H$, when $n \ge 2$
\item $g=g(C)=\frac{1}{2}[K_X H^{n-1}+(n-1)d]+1$ is the sectional genus of $X$.
\item For $1 \le i \le n-1$, let $H_i \in |H|$ be general divisors and set $X_n:=X$ and $X_i=H_1\cap\cdots\cap H_{n-i}$. 
\end{itemize} 
\end{notation}

\section{Generalities on Ulrich bundles}

We collect some well-known facts, to be used sometimes later.

\begin{defi}
Let $\E$ be a vector bundle on $X$. We say that $\E$ is an {\it Ulrich vector bundle} for $(X,H)$ if $H^i(\E(-p))=0$ for all $i \ge 0$ and $1 \le p \le n$.
\end{defi}

We have
 
\begin{lemma}
\label{ulr}
Let $\E$ be a rank $r$ Ulrich vector bundle for $(X,H)$. Then
\begin{itemize}
\item[(i)] $c_1(\E) H^{n-1}=\frac{r}{2}[K_X+(n+1)H] H^{n-1}$.
\item[(ii)] If $n \ge 2$, then $c_2(\E) H^{n-2}=\frac{1}{2}[c_1(\E)^2-c_1(\E) K_X] H^{n-2}+\frac{r}{12}[K_X^2+c_2(X)-\frac{3n^2+5n+2}{2}H^2] H^{n-2}$.
\item[(iii)] $\chi(\E(m))= \frac{rd}{n!}(m+1) \cdots (m+n)$.
\item[(iv)] $H^n(\E(m))=0$ if and only if $m \ge -n$. 
\item[(v)] $\E^*(K_X+(n+1)H)$ is also an Ulrich vector bundle for $(X,H)$. 
\item [(vi)] $\E$ is globally generated.
\item [(vii)] $h^0(\E)=rd$.
\item [(viii)] $\E$ is arithmetically Cohen-Macaulay (aCM), that is $H^i(\E(j))=0$ for $0 < i <n$ and all $j \in \Z$.
\item [(ix)] $\E_{|Y}$ is Ulrich on a smooth hyperplane section $Y$ of $X$.
\end{itemize}
\end{lemma}
\begin{proof} 
We have 
\begin{equation}
\label{can}
K_{X_i}=\left(K_X+(n-i)H\right)_{|X_i}, 1 \le i \le n.
\end{equation}
By \cite[Lemma 2.4(iii)]{ch} we have that
$$c_1(\E) H^{n-1} = \deg(\E_{|C}) =r(d+g-1)$$
and using \eqref{can} on $C=X_1$ we have
$$K_X H^{n-1} = 2(g-1)-(n-1)d$$
thus giving (i). To see (ii) observe that the exact sequences, for $1 \le i \le n-1$, 
$$0 \to T_{X_i} \to (T_{X_{i+1}})_{|X_i} \to H_{|X_i} \to 0$$
and \eqref{can} give by induction that 
\begin{equation}
\label{c2}
c_2(S)=c_2(X_2)=c_2(X) H^{n-2}+(n-2) K_X H^{n-1} + \binom{n-1}{2}d.
\end{equation}
It follows from \cite[Prop.~2.1(2.2)]{c}, \eqref{can}, and Noether's formula $12 \chi(\O_S)-K_S^2=c_2(S)$ that 
\begin{equation*}
\begin{array}{ccc}
c_2(\E) H^{n-2} & = \frac{1}{2}[c_1(\E)^2-c_1(\E)(K_X+(n-2)H)]H^{n-2}-r \left(H^n-\frac{[K_X+(n-2)]^2 H^{n-2}+c_2(S)}{12}\right) = \hskip.75cm \\
& =  \frac{1}{2}[c_1(\E)^2-c_1(\E)K_X]H^{n-2}-\frac{n-2}{2}c_1(\E)H^{n-1}-r\left(H^n-\frac{[K_X+(n-2)H]^2 H^{n-2}+c_2(S)}{12}\right)
\end{array}
\end{equation*}
\normalsize
Now (ii) follows from the above equation by using (i) and \eqref{c2}. Next, (iii) is \cite[Lemma 2.6]{ch}. To see (iv) observe that $\E$ is $0$-regular, hence it is $q$-regular for every $q \ge 0$ and therefore $H^n(\E(q-n))=0$, that is (iv). Also, (v) follows by definition and Serre duality, while (vi) follows by definition, since $\E$ is $0$-regular, and \cite[Thm.~1.8.5]{laz1}. For (vii), (viii) and (ix) see \cite[Prop.~2.1]{es} (or \cite[(3.1)]{b1}) and \cite[(3.4)]{b1}.
\end{proof}

\section{$T_X(k)$ Ulrich in any dimension}

We start by drawing some consequences on $(X,H,k)$, of cohomological and numerical type, when $T_X(k)$ is an Ulrich vector bundle. 

\begin{lemma}
\label{proj}
Let $(X,H)=(\P^n,\O_{\P^n}(1)), n \ge 1$. Then $T_X(k)$ is an Ulrich vector bundle if and only if $n=1$ and $k=-2$. 
\end{lemma}
\begin{proof}
The assertion is obvious if $(X,H,k)=(\P^1,\O_{\P^1}(1),-2)$. Vice versa suppose that $T_X(k)$ is an Ulrich vector bundle. If $(X,H)=(\P^n,\O_{\P^n}(1))$, it follows by \cite[Prop.~2.1]{es} (or \cite[Thm.~2.3]{b1}) that $T_{\P^n}(k) \cong \O_{\P^n}^{\oplus n}$, hence $0 = \det(T_{\P^n}(k))=\O_{\P^n}(nk+n+1)$, so that $1=-n(k+1)$, giving $n=1, k=-2$. 
\end{proof}

\begin{lemma} {\rm (cohomological conditions)} 
\label{coh} 

Let $X \subseteq \P^N$ be a smooth irreducible variety of dimension $n \ge 1$. If $T_X(k)$ is an Ulrich vector bundle we have:
\begin{itemize}
\item[(i)] Either $(X,H,k)=(\P^1,\O_{\P^1}(1),-2)$, or $k \ge 0$.
\item[(ii)] If $n \ge 2$, then $T_X$ is aCM, that is $H^i(T_X(j))=0$ for $1 \le i \le n-1$ and for every $j \in \Z$. In particular $H^i(T_X)=0$ for $1 \le i \le n-1$.
\item[(iii)] If $k \ge 1$, then $H^0(T_X)=0$, hence $X$ has discrete automorphism group.
\item[(iv)] If $n \ge 2$, then $X$ is infinitesimally rigid, that is $H^1(T_X)=0$.
\item[(v)] $H^0(K_X+(n-k-2)H)=0$ and, if $n \ge 2$, also $H^0(K_X+(n-k-1)H)=0$.
\item[(vi)] If $q(X) \ne 0$ then $H^0(K_X+(n-k)H)=0$.
\item[(vii)] If $k \le n-1$, then $p_g(X)=0$.
\item[(viii)] Let $a(X,H)=\min\{l \in \Z : lH-K_X \ge 0\}$. Then $k \le \frac{a(X,H)(n+2)}{2n}+\frac{n+1}{2}$. 

\noindent Moreover $H^0((\lceil \frac{n(2k-n-1)}{n+2} \rceil -1)H-K_X)=0$.
\item[(ix)] $K_X-kH$ is not big. 
\end{itemize} 
\end{lemma}
\begin{proof}
Since $T_X(k)$ is an Ulrich vector bundle, it is globally generated by Lemma \ref{ulr}(vi). Now if $k \le -1$ we would have that $0 \ne H^0(T_X(k)) \subseteq H^0(T_X(-1))$. But then the Mori-Sumihiro-Wahl's theorem \cite[Thm.~8]{ms}, \cite[Thm.~1]{w2} implies that $(X,H)=(\P^1,\O_{\P^1}(2)), (\P^n,\O_{\P^n}(1))$. In the first case we have that $0 = H^i(T_{\P^1}(k-1))=H^i(\O_{\P^1}(2k))=0$ for $i \ge 0$, a contradiction. In the second case apply Lemma \ref{proj}. This proves (i). Now (ii) follows by Lemma \ref{ulr}(viii). If $k \ge 1$ we have that $H^0(T_X) \subseteq  H^0(T_X(k-1))=0$, hence (iii). (iv) is implied by (ii). As for (v), recall that, as is well known, $\Omega^1_X(2)$ is globally generated. Now if $H^0(K_X+(n-k-2)H) \ne 0$ then we get the contradiction
$$0 \ne H^0(\Omega^1_X(2)) \subseteq H^0(\Omega^1_X(K_X+(n-k)H))=H^n(T_X(k-n))^*=0.$$
This gives the first part of (v). Similarly, if $q(X) \ne 0$ and $H^0(K_X+(n-k)H) \ne 0$ then we get the contradiction
$$0 \ne H^0(\Omega^1_X) \subseteq H^0(\Omega^1_X(K_X+(n-k)H))=H^n(T_X(k-n))^*=0.$$
This gives (vi). Now if $n \ge 2$, consider $Y \in |H|$ smooth. Then $T_X(k)_{|Y}$ is an Ulrich vector bundle on $Y$ by Lemma \ref{ulr}(ix), hence $H^{n-1}(T_X(k-n+1)_{|Y})=0$. Now the exact sequence
$$0 \to T_Y(k-n+1) \to T_X(k-n+1)_{|Y} \to \O_Y(k-n+2) \to 0$$
implies that $H^{n-1}(\O_Y(k-n+2))=0$. Hence, setting $\L = K_X+(n-k-1)H$, we get by Serre's duality that
$$H^0(\L_{|Y})=H^0(K_Y+(n-k-2)H_{|Y})=0.$$
Therefore $H^0(\L(-l)_{|Y})=0$ for every $l \ge 0$ and the exact sequences
$$0 \to \L(-l-1) \to \L(-l) \to \L(-l)_{|Y} \to 0$$
show that $h^0(\L(-l-1) )=h^0(\L(-l))$ for every $l \ge 0$. Since this is zero for $l \gg 0$, we get that they are all zero, hence $H^0(K_X+(n-k-1)H)=0$. This proves the second part of (v). Now, to see (vii), suppose that $k \le n-1$. If $n \ge 2$, we see that (v) gives $H^0(K_X) \subseteq H^0(K_X+(n-k-1)H)=0$, hence (vii). If $n=1$ we have that $k \le 0$, hence $X=\P^1$ by (i) and Lemma \ref{posi}(i). Observe that $a(X,H)H-K_X \ge 0$, hence $(a(X,H)H-K_X)H^{n-1} \ge 0$ and using Lemma \ref{posi}(ii), we get 
$$a(X,H) \ge \frac{n(2k-n-1)}{n+2}$$
This gives (viii) since, by its own definition, $H^0((a(X,H) -1)H-K_X)=0$. Finally assume that $K_X-kH$ is big. Then Serre's duality gives $H^0(T_X(k))=H^n(\Omega^1_X(K_X-kH))^*=0$ by Bogomolov-Sommese vanishing \cite[Thm.~4]{bo}, contradicting Lemma \ref{ulr}(vi).
\end{proof}

We recall that a {\it pseudoeffective (or pseff)} divisor, on a variety $X$, is a divisor $D$ whose class lies in the closure of the effective cone in the space $N^1(X)_{\R}$ of numerical equivalence classes of divisors. 

\begin{lemma} {\rm (numerical conditions)}
\label{posi} 

Let $X \subseteq \P^N$ be a smooth irreducible variety of dimension $n \ge 1$. If $T_X(k)$ is an Ulrich vector bundle we have:
\begin{itemize}
\item[(i)] $d= \frac{(n+2)(g-1)}{nk-1}$. In particular either $(X,H,k)=(\P^1,\O_{\P^1}(1),-2)$, or $g=k=0$, or $g \ge 2$. 
\item[(ii)] $k=\frac{n+1}{2}+\left(\frac{n+2}{2nd}\right)K_XH^{n-1}$; equivalently $K_X H^{n-1} = \frac{n(2k-n-1)}{n+2}d$.
\item[(iii)] If $k < \frac{n+1}{2}$, then $X$ is rationally connected and $H^i(\O_X)=0$ for every $i \ge 1$.
\item[(iv)] If $k > \frac{n+1}{2}$, then $-K_X$ is not pseff.
\item[(v)] $T_X$ is semistable.
\item[(vi)] If $n \ge 2$, then $K_X^2 H^{n-2} \le \frac{2n}{n-1} c_2(X)H^{n-2}$.
\item[(vii)] If $n \ge 2$, then 
$$(12kn-12k^2+12k-3n^2-5n-2)nd+2(n+12)K_X^2 H^{n-2}+2(n-12)c_2(X)H^{n-2}=0.$$
\end{itemize} 
\end{lemma}
\begin{proof}
Since $c_1(T_X(k))=-K_X+nkH$, we get by Lemma \ref{ulr}(i) that
$$(-K_X+nkH)H^{n-1} = \frac{n}{2}\left(K_X H^{n-1} + (n+1)d\right)$$
and this gives (ii). Also, using $K_X H^{n-1} = 2(g-1)-(n-1)d$, we get that 
$$(nk-1)d=(n+2)(g-1).$$ 
Now if $nk-1=0$ then $n=k=g=1$, but then $T_X(k)=\O_X(1)$ is not Ulrich. Therefore $nk-1 \ne 0$ and $d= \frac{(n+2)(g-1)}{nk-1}$. Hence $g \ne 1$ and if $g=0$ then either $k=0$ or $k \ne 0$ and in the latter case we have that $nk<1$, hence $(X,H,k)=(\P^1,\O_{\P^1}(1),-2)$ by Lemma \ref{coh}(i). This proves (i). Next, (v) follows since Ulrich vector bundles are semistable by \cite[Thm.~2.9]{ch}, hence Bogomolov's inequality gives (vi). To see (iii), suppose that $k < \frac{n+1}{2}$. If $n \ge 2$, then (ii) gives that $K_X H^{n-1} < 0$, hence $X$ is rationally connected by (v) and \cite[Main Thm.]{bmq} (see also \cite[Thm.~1.1]{cp}). Hence, as is well known, $H^i(\O_X)=0$ for every $i \ge 1$. If $n=1$ then $k \le 0$ and $X=\P^1$ by (i). Thus we get (iii). If $k > \frac{n+1}{2}$, then either $n=1$ and $g \ge 2$ by (i), so that $-K_X$ is not pseff, or $n \ge 2$ and (ii) gives that $K_X H^{n-1} > 0$, hence again $-K_X$ is not pseff and we get (iv). To see (vii), observe that
\begin{equation}
\label{c2t}
c_2(T_X(k))H^{n-2}=c_2(X)H^{n-2} - k(n-1)K_X H^{n-1} + \binom{n}{2}k^2d.
\end{equation}
From Lemma \ref{ulr}(ii), we get 
\begin{equation}
\label{c2n}
c_2(T_X(k))H^{n-2}=\left(\frac{n^2k^2}{2}-\frac{n}{24}\left(3n^2+5n+2\right)\right)d-\frac{3nk}{2}K_X H^{n-1}+\left(1+\frac{n}{12}\right)K_X^2 H^{n-2}+\frac{n}{12}c_2(X)H^{n-2}.
\end{equation}
Combining \eqref{c2t}, \eqref{c2n} and (ii), we obtain (vii).
\end{proof}

\begin{defi}
For $n \ge 1$ we denote by $Q_n$ a smooth quadric in $\P^{n+1}$.
\end{defi}

\begin{lemma}
\label{quad}
Let $(X,H)=(Q_n,\O_{Q_n}(1)), n \ge 1$. Then $T_X(k)$ is not an Ulrich vector bundle for any integer $k$. 
\end{lemma}
\begin{proof}
This follows from the well-known fact that Ulrich vector bundles on quadrics are direct sums of spinor bundles. Alternatively, since $g=0$, it follows by Lemma \ref{posi}(i) that $k=0$ and $2 = d = n+2$, a contradiction.
\end{proof}

We will use the nef value of $(X,H)$:
\begin{equation}
\label{tau}
\tau(X,H) = \min\{t \in \R : K_X+tH \ \hbox{is nef}\}.
\end{equation}
We observe that in \cite[Def.~1.5.3]{bs} the nef value is defined only when $K_X$ is not nef. On the other hand, it makes sense and it will be used, throughout this paper, also when $K_X$ is nef.

A very useful observation is the following.

\begin{lemma}
\label{gg}
Let $X \subseteq \P^N$ be a smooth irreducible variety of dimension $n \ge 1$. If $T_X(k)$ is an Ulrich vector bundle, then $\Omega^1_Y({K_X}_{|Y}+(n+1-k)H_{|Y})$ is globally generated for any smooth subvariety $Y \subseteq X$. Moreover: 
\begin{itemize}
\item[(i)] If $\pm (K_X + \frac{n(n+1-2k)}{n+2}H)$ is pseff, then $K_X \equiv \frac{n(2k-n-1)}{n+2}H$.
\item[(ii)] $\tau(X,H) \ge \frac{n(n+1-2k)}{n+2}$.
\item[(iii)] $\tau(X,H) \le n-\frac{nk}{n+1}$. In particular, if $K_X$ is not nef, then $k \le n$.
\item[(iv)] If $k \ge n+1$, then $K_X$ is ample.
\end{itemize}
\end{lemma}
\begin{proof}
Note that $\Omega^1_X(K_X+(n+1-k)H)$ is Ulrich and globally generated by Lemma \ref{ulr}(v) and (vi). Since $\Omega^1_X(K_X+(n+1-k)H)$ surjects onto $\Omega^1_Y({K_X}_{|Y}+(n+1-k)H_{|Y})$, the latter is also globally generated. Moreover so is $\det(\Omega^1_X(K_X+(n+1-k)H)=(n+1)K_X+n(n+1-k)H$, hence we get (iii) and, if $k \ge n+2$, we also deduce that $K_X$ is ample. On the other hand, if $k=n+1$, then $\Omega^1_X(K_X)$ is Ulrich,  and we claim that $\det(\Omega^1_X(K_X))=(n+1)K_X$ is ample. In fact, if not, then \cite[Thm.~1]{ls} implies that there is a line $L \subset X$ such that $\Omega^1_X(K_X)_{|L}$ is trivial. Hence $(n+1)K_X \cdot L = \deg(\Omega^1_X(K_X)_{|L})=0$. But then we have a surjection $\Omega^1_X(K_X)_{|L} \to \Omega^1_L$, contradicting the fact that $\Omega^1_L$ is not globally generated. This proves (iv). As for (i) and (ii), set $q=\frac{n(n+1-2k)}{n+2}$, so that $(K_X+qH)H^{n-1}=0$ by Lemma \ref{posi}(ii). Now if $\pm (K_X + qH)$ is pseff, then $K_X + qH \equiv 0$ by  \cite[Cor. 3.15]{fl2} (see also \cite[Prop. 3.7]{fl1}), thus proving (i). Also, (i) implies that either $K_X \equiv -qH$ and then $\tau(X,H) =q$, or $K_X + qH$ is not pseff, hence not nef. Therefore, in the latter case, $\tau(X,H) > q$, proving (ii).
\end{proof}

In the case of hypersurfaces we have

\begin{lemma}
\label{hyper}

Let $X \subset \P^{n+1}$ be a smooth irreducible nondegenerate hypersurface. Then $T_X(k)$ is not an Ulrich vector bundle.
\end{lemma}
\begin{proof} 
We have $d \ge 2$. If $n=1$ we have that $K_X=(d-3)H$, $g=\binom{d-1}{2}$ and $k= \frac{3(d-3)}{2}+1$ by Lemma \ref{posi}(i). Now $0 = H^0(T_X(k-1))=H^0((-d+2+k)H)$ and therefore $-d+2+k \le -1$, giving  the contradiction $d \le 1$. Hence $n \ge 2$ and since $C \subset \P^2$ we have that $g-1=\frac{d(d-3)}{2}$ and Lemma \ref{posi}(i) implies that $d= \frac{2(nk-1)}{n+2}+3$. On the other hand Lemma \ref{coh}(v) gives that
$$0 = H^0(K_X+(n-k-1)H)=H^0((d-k-3)H)$$
and therefore
$$\frac{2(nk-1)}{n+2}-k \le -1$$
that is $k(n-2)+n \le 0$, a contradiction.
\end{proof}
 
\begin{lemma}
\label{bigbound}

Let $X \subseteq \P^N$ be a smooth irreducible variety of dimension $n \ge 1$. If $T_X(k)$ is an Ulrich vector bundle we have that
$$k \le \frac{(n+2)(d-4)+4}{4n}.$$
\end{lemma}
\begin{proof} 
We have $X \subseteq \P H^0(H)=\P^N$. If $N=n$ then $(X,H)=(\P^n,\O_{\P^n}(1))$ and Lemma \ref{proj} gives that $n=1$ and $k=-2$. Since $d=1$ we have that $k = -2 \le -\frac{5}{4}= \frac{(n+2)(d-4)+4}{4n}$. If $N > n$, by Lemma \ref{hyper} we have that $N \ge n+2$  and $C \subset \P^{N-n+1}$ can be projected isomorphically to a non-degenerate smooth irreducible curve in $\P^3$. Then Castelnuovo's bound gives that $g - 1 \le \frac{d(d-4)}{4}$ and Lemma \ref{posi}(ii) implies the required bound on $k$. 
\end{proof}

A nice consequence of the above lemmas is the following.

\begin{prop}
\label{canzero}

There does not exist any $(X,H,k)$ with $T_X(k)$ an Ulrich vector bundle, when:
\begin{itemize}
\item[(i)] $K_X \equiv 0$. 
\item[(ii)] $\pm K_X$ is pseff and $k=\frac{n+1}{2}$.
\end{itemize} 
\end{prop}
\begin{proof} 
Under hypothesis (ii), we get from Lemma \ref{gg}(i) that $K_X \equiv 0$. Thus we will be done if we prove (i). Assume next that $K_X \equiv 0$, so that $k=\frac{n+1}{2}$ by Lemma \ref{posi}(ii) and $n \ge 3$ by Lemma \ref{posi}(i). Since $H-K_X$ is ample, it follows by Kodaira vanishing that $H^i(H)=H^i(K_X+H-K_X)=0$ for $i > 0$, hence 
$$h^0(K_X+H)=\chi(K_X+H)=\chi(H)=h^0(H) \ne 0.$$ 
On the other hand, Lemma \ref{coh}(v) gives that $h^0(K_X+\frac{n-3}{2}H)=0$, whence, if $n \ge 5$, we get the contradiction $h^0(K_X+H)=0$. 

It remains to consider the case $n=3, k=2$. Note that $p_g(X)=0$ by Lemma \ref{coh}(vii) and $q(X)=0$, for otherwise Lemma \ref{coh}(vi) gives that $h^0(K_X+H)=0$. Therefore $\chi(\O_X) \ge 1$. On the other hand $\chi(\O_X)=\frac{1}{24}c_1(X)c_2(X)=0$, a contradiction.
\end{proof}

We now prove Theorem \ref{bou}.

\renewcommand{\proofname}{Proof of Theorem \ref{bou}}
\begin{proof}
By the Hodge index theorem we have that $H_{|S}^2K_S^2 \le (H_{|S}K_S)^2$, that is
\begin{equation}
\label{hit}
d K_X^2 H^{n-2} \le (K_X H^{n-1})^2. 
\end{equation}
Using Lemma \ref{posi}(vi), (vii) and \eqref{hit} we obtain that
$$0= (12kn-12k^2+12k-3n^2-5n-2)nd+2(n+12)K_X^2 H^{n-2}+2(n-12)c_2(X)H^{n-2} \le$$
$$\le (12kn-12k^2+12k-3n^2-5n-2)nd+\frac{3n^2+11n+12}{n}K_X^2 H^{n-2} \le$$
$$\le (12kn-12k^2+12k-3n^2-5n-2)nd+\frac{3n^2+11n+12}{nd}(K_X H^{n-1})^2$$
which, using Lemma \ref{posi}(ii) becomes
$$4nk^2-4n(n+1)k-3n^2-7n-4 \le 0$$
giving
$$k \le \frac{n^2+n+\sqrt{n^4+5n^3+8n^2+4n}}{2n} < n+2.$$
\end{proof}
\renewcommand{\proofname}{Proof}

The case $k=0$ is known:

\begin{thm} (\cite[Prop.~4.1, Thm.~4.5]{bmpt})
\label{k=0}

Let $X \subseteq \P^N$ be a smooth irreducible variety of dimension $n \ge 1$. Then $T_X$ is an Ulrich vector bundle if and only if $(X,H)=(\P^1,\O_{\P^1}(3)), (\P^2,\O_{\P^2}(2))$.
\end{thm}

We will give a quick alternative proof in section \ref{sez7}.

Next we study the case when $K_X$ and $H$ are proportional.

\begin{thm}
\label{prop}

Let $X \subseteq \P^N$ be a smooth irreducible variety of dimension $n \ge 1$. Suppose that the numerical classes of $H$ and $K_X$ are proportional and that either 
\begin{itemize}
\item[(i)] $1 \le n \le 11$ and either $k \le \frac{n+1}{2}$ or $k \ge n+2$; or 
\item[(ii)] $n=12$, or
\item[(iii)] $n \ge 13$ and $k \not \in \{n+2,n+3\}.$
\end{itemize} 
Then $T_X(k)$ is Ulrich if and only if $(X,H,k)$ is one of the following:
\begin{itemize}
\item[(1)] $(\P^1,\O_{\P^1}(1),-2)$,
\item[(2)] $(\P^1,\O_{\P^1}(3),0)$,
\item[(3)] $(\P^2,\O_{\P^2}(2),0)$, 
\item[(4)] $(S_5,-2{K_{S_5}},1)$, where $S_5$ is a Del Pezzo surface of degree $5$.
\end{itemize} 
\end{thm}
\begin{proof} 
In the cases (1)-(4) we have that $T_X(k)$ is Ulrich by Lemma \ref{posi}(i), Theorem \ref{k=0} and Theorem \ref{k=1sup}. 

Vice versa, suppose that the numerical classes of $H$ and $K_X$ are proportional and that we are under one of hypotheses (i), (ii) or (iii) and that $T_X(k)$ is Ulrich. 

Observe that, since $N^1(X)$ is a torsion free finitely generated abelian group, we can find an ample primitive (that is indivisible)  divisor $A$ and some $r, s \in \Z$ such that $s>0, H \equiv sA$ and $K_X \equiv -rA$. In particular, Lemma \ref{posi}(ii) gives
\begin{equation}
\label{eq1}
r(n+2)=n(n+1-2k)s.
\end{equation}

If $k \le 0$ we are in cases (1)-(3) by Lemma \ref{coh}(i) and Theorem \ref{k=0}. Hence assume that $k \ge 1$. Note that Proposition \ref{canzero} shows that $k \ne \frac{n+1}{2}$. 

If (ii) or (iii) holds, since the numerical classes of $H$ and $K_X$ are proportional, Lemma \ref{posi}(vi), (ii) and (vii) imply that
$$4nk^2-4n(n+1)k-3n^2-7n-4 \ge 0$$
so that $k > n+1$. This is a contradiction under hypothesis (ii) by Theorem \ref{bou}. Under hypothesis (iii), we get that $k \ge n+4$. Then it follows by \eqref{eq1} that
$$-r-ks=\frac{(n-2)k-n^2-n}{n+2}s \ge \frac{n-8}{n+2}s>0$$
hence $K_X-kH=(-r-ks)A$ is ample, contradicting Lemma \ref{coh}(ix). Thus it remains to consider hypothesis (i).

Now assume (i), so that $n \ge2$ and Theorem \ref{bou} implies that it cannot be that $k \ge n+2$. Hence $k < \frac{n+1}{2}$ and Lemma \ref{posi}(ii) implies that $X$ is Fano. Consequently, the numerical and linear equivalence for divisors coincide on $X$ and then $K_X=-rA$ and $H=sA$. We can also assume that $r \le n-1$, for otherwise, as is well known, $(X,H)=(\P^n,\O_{\P^n}(1)), (Q_n,\O_{Q_n}(1))$, contradicting Lemmas \ref{proj} and \ref{quad}.

Next, set $P_{A}(t):=\chi(K_X+tA)$, so that $P_A(t) = h^0(K_X+tA)$ whenever $t \ge 1$ is an integer. By Riemann-Roch (see for example \cite[eq.~(1), p.~2]{ho}, we have
\begin{equation}
\label{rr}
P_{A}(t)=\frac{A^n}{n!}t^n+\frac{A^{n-1}K_X}{2(n-1)!}t^{n-1}+\frac{A^{n-2}(K_X^2+c_2(X))}{12(n-2)!}t^{n-2}+\cdots+(-1)^n\chi(\O_X).
\end{equation}

Note that $n$ is even, for otherwise $n$ and $n+2$ are coprime, and consequently $n$ divides $r$ by \eqref{eq1}, hence $r \ge n$, a contradiction. 

Set $n=2m$, where $1 \le m \le 5$.

If $m=1$ we have that $n=2, k=1$ and we are in case (4) by Theorem \ref{k=1sup}. 

We will now exclude the remaining cases for $m$.

Note that we can rewrite \eqref{eq1} as 
\begin{equation}
\label{id}
(m+1)r=m(2m-2k+1)s.
\end{equation}

\noindent\underline{Case 1: $m=2$.} In this case $k \le 2$ and $r \le 3$. 

{\it (1a): $k=1$.} We see that $r=2s$. Thus $(r,s)=(2,1)$, contradicting Lemma \ref{coh}(v).

{\it (1b): $k=2$.} We see that $2s=3r$. Thus $(r,s)=(2,3)$ and Lemma \ref{coh}(v) shows that $h^0(A)=0$. This contradicts \cite[Lemma 2]{a} or \cite[Thm.~5.1]{k}.

\noindent\underline{Case 2: $m=3$.} In this case $k \le 3$ and $r \le 5$. 

{\it (2a): $k=1$.} We see that $4r=15s$. Thus, $15$ divides $r$ which is clearly impossible.

{\it (2b): $k=2$.} We see that $9s=4r$. Thus, $9$ divides $r$ which is a contradiction. 

{\it (2c): $k=3$.} We see that $3s=4r$. Thus $(r,s)=(3,4)$ and Lemma \ref{coh}(v) shows that $h^0(K_X+8A)=0$. This is a contradiction by \cite[Thm.~1.2]{gl}, as $K_X+8A$ is base-point-free.

\noindent\underline{Case 3: $m=4$.} In this case $k \le 4$ and $r \le 7$. 

{\it (3a): $k=1$.} We see that $5r=28s$. Thus, $28$ divides $r$ which is clearly impossible.

{\it (3b): $k=2$.} We see that $4s=r$. Thus $(r,s)=(4,1)$ and Lemma \ref{coh}(v) shows $h^0(K_X+5A)=h^0(H)=0$, a contradiction.  

{\it (3c): $k=3$.} We see that $12s=5r$. Thus, $12$ divides $r$ which is absurd.

{\it (3d): $k=4$.} We see that $4s=5r$. Thus, $(r,s)=(4,5)$. 

We have $K_X=-4A$ and $H=5A$. Hence $H^0(K_X+15A) = 0$ by Lemma \ref{coh}(v). Then
$$P_{A}(1)=P_A(2)=P_A(3)=P_A(5)=P_A(10)=P_A(15)=0,\quad P_A(0)=P_A(4)=1$$
and
$$P_A(t)=\frac{A^8}{8!}(t-1)(t-2)(t-3)(t-5)(t-10)(t-15)(t^2+at+b).$$
Therefore
\begin{equation}
\label{eq2}
1=P_A(0)=\frac{A^8}{8!}(4500b) \implies \frac{A^8}{8!}b=\frac{1}{4500}
\end{equation}
and calculating the coefficient of $t^7$ in \eqref{rr} we get 
\begin{equation}\label{eq3}
\frac{A^8}{8!}(a-36)=\frac{A^7 K_X}{2(7!)}\implies a=20.
\end{equation}
We also know that $P_A(4)=1$ and that gives us 
$$-\frac{A^8}{8!}(396)(16+4a+b)=1.$$
We simplify the above using \eqref{eq2} and \eqref{eq3} to obtain
$$-38016\frac{A^8}{8!}=1+\frac{396}{4500}$$
which is clearly absurd.

\noindent\underline{Case 4: $m=5$.} In this case $k \le 5$ and $r \le 9$. 

{\it (4a): $k=1$.} We see that $2r=15s$. Thus, $15$ divides $r$ which is impossible.

{\it (4b): $k=2$.} We see that $35s=6r$. Thus, $35$ divides $r$ which is also impossible.

{\it (4c): $k=3$.} We see that $25s=6r$. Thus, $25$ divides $r$ which is also impossible.

{\it (4d): $k=4$.} We see that $5s=2r$. Thus $(r,s)=(5,2)$ and Lemma \ref{coh}(v) shows that $h^0(K_X+10A)=0$. Then
$$P_{A}(1)=P_A(2)=P_A(3)=P_A(4)=P_A(6)=P_A(8)=P_A(10)=0,\quad P_A(0)=P_A(5)=1$$
so that we obtain 
$$P_A(t)=\frac{A^{10}}{10!}(t-1)(t-2)(t-3)(t-4)(t-6)(t-8)(t-10)(t^3+at^2+bt+c).$$
Therefore
\begin{equation}
\label{eq4}
1=P_A(0)=-\frac{A^{10}}{10!}(11520c)=1\implies \frac{A^{10}}{10!}c=-\frac{1}{11520}.
\end{equation}
and calculating the coefficient of $t^9$ in \eqref{rr} we get 
\begin{equation}
\label{eq5}
\frac{A^{10}}{10!}(a-34)=\frac{A^9 K_X}{2(9!)}\implies a=9.
\end{equation}
We also know that $P_A(5)=1$ and that gives us 
$$-\frac{A^{10}}{10!}(360)(125+25a+5b+c)=1.$$
We simplify the above using \eqref{eq4} and \eqref{eq5} to obtain
\begin{equation}
\label{eq6}
\frac{A^{10}}{10!}b=\frac{1}{5(11520)}-\frac{1}{5(360)}-70\frac{A^{10}}{10!}.
\end{equation}
Finally, calculating the coefficient of $t^8$ in \eqref{rr} we get 
\begin{equation}
\label{eq7}
\frac{A^{10}}{10!}(b-34a+463)=\frac{A^8(K_X^2+c_2(X))}{12(8!)}.
\end{equation}
We simplify \eqref{eq7} using \eqref{eq5} and \eqref{eq6} to obtain 
\begin{equation}
\label{rev1}
67A^{10}+5A^8 c_2(X)+1302=0.
\end{equation}
On the other hand, Lemma \ref{posi}(vii) shows that 
$$115A^{10}=A^8c_2(X)$$
and combining with \eqref{rev1}, we get that $A^{10}$ is negative, which is clearly impossible.

{\it (4e): $k=5$.} We see that $5s=6r$. Thus $(r,s)=(5,6)$. We have $K_X=-5A$ and $H=6A$. Again $H^0(K_X+24A) = 0$ by Lemma \ref{coh}(v). Then
$$P_{A}(1)=P_A(2)=P_A(3)=P_A(4)=P_A(6)=P_A(12)=P_A(18)=P_A(24)=0,\quad P_A(0)=P_A(5)=1.$$
Thus, we obtain 
$$P_A(t)=\frac{A^{10}}{10!}(t-1)(t-2)(t-3)(t-4)(t-6)(t-12)(t-18)(t-24)(t^2+at+b).$$
Then
\begin{equation}
\label{eq8}
1 = P_A(0)=\frac{A^{10}}{10!}(746496b) \implies \frac{A^{10}}{10!}b=\frac{1}{746496}.
\end{equation}
calculating the coefficient of $t^9$ in \eqref{rr} we get 
\begin{equation}
\label{eq9}
\frac{A^{10}}{10!}(a-70)=\frac{A^9K_X}{2(9!)}\implies a=45.
\end{equation}
We also know that $P_A(5)=1$ and that gives us 
$$\frac{A^{10}}{10!}(41496)(25+5a+b)=1.$$
We simplify the above using \eqref{eq8} and \eqref{eq9} to obtain
$$A^{10}=\left(\frac{705000}{746496}\right)(10!)$$
which is clearly absurd since $A^{10}$ is an integer.
\end{proof}

\begin{cor}
\label{sottoc}
Suppose that $K_X=eH, e \in \ZZ$ (hence, in particular, if $\Pic(X)=\Z H$). Then $T_X(k)$ is Ulrich if and only if $(X,H,k)=(\P^1,\O_{\P^1}(1),-2)$.
\end{cor}
\begin{proof} 
This follows by \cite[Prop.~4.1(i)]{lo}. We give another proof. We have that $e = \frac{n(2k-n-1)}{n+2}$ by Lemma \ref{posi}(ii). If $k \le \frac{n+1}{2}$ it follows by Theorem \ref{prop} that $(X,H,k)=(\P^1,\O_{\P^1}(1),-2)$. Now assume that $k > \frac{n+1}{2}$, so that $e \ge 1$. If $n \ge 2$, Lemma \ref{coh}(v) gives that $k \ge n+e$, hence $k(n-2)+n \le 0$, a contradiction. Then $n=1$ and $e = \frac{2(k-1)}{3}$. But $0=H^0(T_X(k-1))=H^0((k-1-e)H)$, hence $e \ge k$, so that $k \le -2$, contradicting $k > 1$.
\end{proof}

\begin{cor}
\label{altroprop}
Let $X \subseteq \P^N$ be a smooth irreducible variety of dimension $n \ge 2$ with $T_X(k)$ is Ulrich. Suppose that there is an ample line bundle $A$ on $X$ such that $K_X=rA$, $H=sA$ for some $r, s \in \Z$ (hence, in particular, if $\Pic(X) \cong \Z$). Let $m(H,A):=\min\{m \ge 0 : H^0(mH+qA) \ne 0 \ \hbox{for all} \ q \ge 1 \}$. Then:

\begin{itemize}
\item[(i)] $m(H,A) > \frac{(n-2)k-2}{n+2}$ and, if $n \ge 3$, then $k < \frac{(n+2)m(H,A)+2}{n-2}$.
\item[(ii)] If $A$ is effective, then $n=2$.
\item[(iii)] If $m(H,A) \le n-3$, then $n \le 11$. 
\end{itemize}
\end{cor}
\begin{proof} 
Observe first that, if $n \ge 3$, then $(n-2)k-2 \ge k-2 \ge 0$: In fact if $k \le 1$ we have a contradiction by Theorem \ref{prop}. Now Lemma \ref{posi}(ii) implies that 
\begin{equation}
\label{erre}
r(n+2)=n(2k-n-1)s
\end{equation}
and Lemma \ref{coh}(v) gives
\begin{equation}
\label{annu}
0 = H^0(K_X+(n-k-1)H)=H^0(\frac{(nk-2k-2)s}{n+2}A).
\end{equation}

To see (i), notice that it is obvious for $n=2$, for $m(H,A) \ge 0$ by definition. If $n \ge 3$ we see by \eqref{annu} that $\frac{(nk-2k-2)s}{n+2} \in \Z$ and we can write $\frac{(nk-2k-2)s}{n+2}=as+b$ for some $a, b \in \Z$ with $a \ge 0, 0 \le b < s$. Since $H^0(aH+bA)=0$ by \eqref{annu}, we get that 
$$\frac{(n-2)k-2}{n+2}-1 < a \le m(H,A) -1$$
giving (i). Now suppose that $A$ is effective. If $n \ge 3$ we know that $(n-2)k-2 \ge 0$, contradicting \eqref{annu}. This proves (ii). To see (iii), notice that if $n \ge 12$, then $k \ge n+2$ by Theorem \ref{prop}(ii) and (iii). Hence $\frac{(n-2)k-2}{n+2} > n-3$ and \eqref{annu} gives that
$$0 = H^0(\frac{(nk-2k-2)s}{n+2}A)=H^0((n-3)H+qA)$$
for some $q \ge 1$, contradicting the hypothesis $m(H,A) \le n-3$. 
\end{proof}

\section{Curves}

Throughout this section we will have that $X \subseteq \P^N$ is a smooth irreducible curve.

It follows by Lemma \ref{posi}(i) and Theorem \ref{k=0} that if $n=1$ and $T_X(k)$ is an Ulrich line bundle, then $(X,H,k)=(\P^1,\O_{\P^1}(1),-2), (\P^1,\O_{\P^1}(3),0)$ or $k \ge 2$ and $g \ge 2$.

We will give below examples with $k=2, 3$, essentially on any curve. Then we will give a sharp bound on $k$ depending on the genus.

The case $k=2$ can be characterized. 

\begin{lemma} 
\label{k=2curve}
Let $X \subseteq \P^N$ be a smooth irreducible curve. Then $T_X(2)$ is Ulrich if and only if there exists $M \in \Pic(X)$ such that $H^i(M)=0$ for $i \ge 0$ and $H=K_X+M$ is very ample.
\end{lemma}
\begin{proof}
If $T_X(2)$ is Ulrich, set $M = H-K_X$. Then $H^i(M)=H^i(T_X(1))=0$ for $i \ge 0$ and $K_X+M=H$ is very ample. Vice versa, let $M \in \Pic(X)$ be such that $H^i(M)=0$ for $i \ge 0$ and $H=K_X+M$ is very ample. Then $H^i(T_X(1))=H^i(M)=0$ for $i \ge 0$, so that $T_X(2)$ is Ulrich.
\end{proof}

\begin{example}
\label{k=2curve-bis}
The case $k=2$ occurs on a curve $X$ if and only if $g \ge 3$.
\end{example}
\begin{proof}
If $T_X(2)$ is Ulrich, we have by Lemma \ref{posi}(i) that $g \ge 3$ unless $g=2, d=3$. But the latter case does not occur, as there is no smooth curve of degree $3$ and genus $2$ in $\P^N$.

Suppose that $g \ge 3$ and let $M \in \Pic(X)$ be such that $H^i(M)=0$ for $i \ge 0$. We claim that $H:=K_X+M$ is very ample. In fact $\deg M = g-1$ by Riemann-Roch, hence $\deg H = 3g-3$. If $g \ge 4$, then $\deg H \ge 2g+1$, hence $H$ is very ample. If $g=3$ we have that $\deg H=2g$ and, as is well known, $H$ is very ample unless $H = K_X+P+Q$ for two points $P,Q \in X$. But then $M=P+Q$ is effective, a contradiction. Thus the claim is proved and, since $H=K_X+M$, we have $H^i(T_X(1))=H^i(M)=0$ for $i \ge 0$, so that $T_X(2)$ is Ulrich.
\end{proof}

\begin{example}
\label{k=3curve}
The case $k=3$ occurs on any curve $X$ with (necessarily) odd genus $g \ge 9$. This was suggested to us by E. Sernesi, whom we thank.
\end{example}
\begin{proof}
Let $d=\frac{3(g-1)}{2}$. We claim that a general $H \in \Pic^d(X)$ is very ample. In fact, first observe that $H^1(H)=0$, for otherwise $K_X-H \ge 0$. But $K_X-H$ is a general line bundle of degree $\frac{g-1}{2} \le g-1$, hence $h^0(K_X-H)=0$. Now, if $H$ were not very ample, there will be two points $p, q \in X$ such that $h^0(H-p-q) \ge h^0(H)-1$. But this can be rewritten, by Riemann-Roch, as $h^1(H-p-q) \ge 1$, that is $K_X-H+p+q \ge 0$. Hence there are some points $p_1, \ldots, p_{\frac{g+3}{2}} \in X$ such that  
$$K_X-H+p+q \sim p_1+ \ldots + p_{\frac{g+3}{2}}$$
that is
$$H \sim K_X-p_1- \ldots - p_{\frac{g+3}{2}}+p+q.$$
This means that $H$ is in the image of the morphism $h: X^{\frac{g+7}{2}} \to \Pic^d(X)$ sending $(p_1,\ldots, p_{\frac{g+3}{2}},p,q)$ to $K_X-p_1- \ldots - p_{\frac{g+3}{2}}+p+q$. But $\dim \Im h \le \frac{g+7}{2} < g$, contradicting that $H$ is general. This proves that there is a non-empty open subset $W$ of $\Pic^d(X)$ such that any $H \in W$ is very ample.

Consider the surjective morphism $\psi: \Pic^d(X) \to \Pic^{3g-3}(X)$ given by $\psi(L)=2L$ and the isomorphism $\varphi: \Pic^{3g-3}(X) \to \Pic^{g-1}(X)$ given by $\varphi(L)=L-K_X$. Let $U$ be the non-empty open subset of $\Pic^{g-1}(X)$ such that $H^i(M)=0$ for $i \ge 0$ for any $M \in U$. Now let $H \in \psi^{-1}(\varphi^{-1}(U)) \cap V \cap W$. Then $H$ is very ample and $2H = K_X+M$. In the embedding given by $H$ we have that $H^i(T_X(2))=H^i(-K_X+2H)=H^i(M)=0$ for $i \ge 0$, hence $T_X(3)$ is Ulrich.
\end{proof}

\begin{remark}
If $k \ge 1$ and $g-1$ is a prime number, then $k \in \{2,4\}$.
\end{remark}
\begin{proof}
By Lemma \ref{posi}(i) we get that $(k-1)d=3(g-1)$ and $g \ge 2$, hence $d \ge 4$. If $3$ does not divide $k-1$ we get that $3$ divides $d$ and $(k-1)\frac{d}{3}=g-1$, so that $k=2$. If $3$ divides $k-1$ we get that $\frac{k-1}{3}d=g-1$, so that $k=4$. 
\end{proof}

\begin{example}
Every odd $k \ge 3$ occurs.
\end{example}
\begin{proof}
Let $E$ be an elliptic curve, let $D$ be a divisor of degree $3$ on $E$ and let $S = E \times \P^1$ with two projections $\pi_1 : S \to E, \pi_2 : S \to \P^1$. Set $C_0 = \pi_2^*(\O_{\P^1}(1))$. Then $H = C_0 + \pi_1^*D$ is very ample on $S$. Let $M \in \Pic^0(E)$ be not $2$-torsion and let $B = \frac{k-1}{2}D + M$. Again $H_1:=(k+2)C_0 + \pi_1^*B$ is very ample on $S$. Set 
$$\L = -K_S + (k-1)H = (k+1)C_0 + \pi_1^*(2B-2M)$$
so that
$$\L-H_1 = -C_0 + \pi_1^*(B-2M)$$
while
$$\L-2H_1 = -(k+3)C_0 + \pi_1^*(-2M)$$
and it is easily seen by the K\"unneth formula, that $H^i(\L-pH_1)=0$ for $i \ge 0, 1 \le p \le 2$. Hence, if $X \in |H_1|$ is a smooth irreducible curve, the exact sequence
$$0 \to \L - 2H_1 \to \L-H_1 \to T_X(k-1) \to 0$$
shows that $H^i(T_X(k-1))=0$ for $i \ge 0$, that is $T_X(k)$ is an Ulrich line bundle.
\end{proof}

We now give a bound for $k$.

We first analyze a special case. We use the notation $(a;b_1,b_2,b_3,b_4,b_5,b_6) \in \Z^7$ for the divisor $a\varepsilon^*L -\sum_{i=1}^6 b_iE_i$ on a smooth cubic $W \subset \P^3$, where $\varepsilon : W \to \P^2$ is the blow up in six points, no three collinear and not on a conic, with exceptional divisors $E_i$ and $L$ is a line in $\P^2$.

\begin{lemma}
\label{632}
Let $X \subset \P^3$ be a smooth irreducible curve of genus $3$ and degree $6$ lying on a smooth cubic $W$. Then $T_X(2)$ is Ulrich if and only if $X$ is linearly equivalent to one of the following divisors on $W$:
\begin{equation}
\label{cla}
(4;1,1,1,1,1,1),(5;2,2,2,1,1,1),(6;3,2,2,2,2,1),(7;3,3,3,2,2,2),(8;3,3,3,3,3,3).
\end{equation}
\end{lemma}
\begin{proof}
Let $D =-K_W$. We have that $D \cdot X=6$ and $X^2=10$. Thus Riemann-Roch gives that $\chi(2D-X)=0$. Further, $D(2D-X)=0$, whence $H^0(2D-X)=0$, for otherwise $X \sim 2D$ and then $X^2=12$, a contradiction. Also, $D(-3D+X)=-3$, whence $H^0(-3D+X)=0$, which, be Serre's duality, is $H^2(2D-X)=0$. Thus, also $H^1(2D-X)=0$. Now the exact sequence
$$0\to 2D-2X\to 2D-X \to T_X(1) \to 0$$
gives that 
$$h^1(T_X(1)) = h^2(2D-2X)=h^0(3K_W+2X)$$ 
by Serre's duality. Since $\deg T_X(1) = 2$, we deduce by Riemann-Roch, that $T_X(2)$ is Ulrich if and only if $H^1(T_X(1))=0$, hence
\begin{equation}
\label{ridu}
T_X(2) \ \hbox{is Ulrich if and only if} \ H^0(3K_W+2X)=0.
\end{equation}
Let $(a;b_1,\cdots,b_6)$ with $b_1 \ge b_2 \ge \cdots \ge b_6 \ge 0$ be the class of $X$. It follows from the assumption on degree and genus that \eqref{0.1.1} holds, whence $X$ is as in \eqref{cla} by Lemma \ref{632num}. Using \eqref{ridu}, it remains to show that $H^0(3K_W+2X)=0$ in all of these cases. 

In case $(4;1,1,1,1,1,1)$, we have that $3K_W+2X=(-1;1,1,1,1,1,1)$ is clearly not effective.

In case $(5;2,2,2,1,1,1)$, we have that $3K_W+2X=(1;-1,-1,-1,1,1,1)$. If it were effective, then so would be $(1;0,0,0,1,1,1)$, a contradiction since no three blown-up points are collinear.

In case $(6;3,2,2,2,2,1)$,  we have that $3K_W+2X=(3;3,1,1,1,1,-1)$. Assume it is effective. Intersecting with $(1;1,1,0,0,0,0)$ we see that $(2;2,0,1,1,1,-1)$ must be effective. Intersecting the latter with $(1;1,0,1,0,0,0)$ we conclude that $(1;1,0,0,1,1,-1)$ must be effective, hence also $(1;1,0,0,1,1,0)$, a contradiction since no three blown-up points are collinear.

In case $(7;3,3,3,2,2,2)$,  we have that $3K_W+2X=(5;3,3,3,1,1,1)$. Assume it is effective. Intersecting with $(1;1,1,0,0,0,0)$ we see that $(4;2,2,3,1,1,1)$ must be effective. Intersecting the latter with $(1;1,0,1,0,0,0)$ we conclude that $(3;1,2,2,1,1,1)$ must be effective. Finally, intersecting with $(1;0,1,1,0,0,0)$ we conclude that $(2;1,1,1,1,1,1)$ must be effective, a contradiction since the blown-up points do not lie on a conic.

In case $(8;3,3,3,3,3,3)$,  we have that $3K_W+2X=(7;3,3,3,3,3,3)$. Observe that $D(3K_W+2X)=3$. Thus, if $3K_W+2X$ were effective, it would contain a divisor $\Gamma$ that is either irreducible, or is a union of three lines, or is a union of a line and a conic. Now $p_a(\Gamma)=-3$, hence $\Gamma$ is not irreducible. On the other hand, since the first coefficient of a line in $W$ is at most $2$ and of a conic at most $3$, we see that $\Gamma$, whose first coefficient is $7$, is not a union of three lines nor of a line and a conic. This contradiction shows that $3K_W+2X$ is not effective.
\end{proof}

Now we give the sharp bound.

\renewcommand{\proofname}{Proof of Theorem \ref{gen}}
\begin{proof}
First assume that $g=0$. Then Lemma \ref{posi}(i) gives that $k \le 0$, that is the required bound, and if equality holds, then Theorem \ref{k=0} shows that $X$ is a curve of type $(1,2)$ on a smooth quadric.

Thus, by Lemma \ref{posi}(i), we can now assume that $g \ge 2$.

Observe that $h^0(H) \ge 4$. In fact, the only possibility remaining is that $h^0(H) =3$. But then $K_X=(d-3)H$, $g=\binom{d-1}{2}$ and $k= \frac{3(d-3)}{2}+1$ by Lemma \ref{posi}(i). Now $0 = H^0(T_X(k-1))=H^0((-d+2+k)H)$ and therefore $-d+2+k \le -1$, giving  the contradiction $d \le 1$. 

Now, if $X$ has general moduli, since it has a $g^3_d$, the Brill-Noether theorem implies that $\rho(g,3,d) \ge 0$, that is $d \ge \frac{3g+12}{4}$. By Lemma \ref{posi}(i) we get that $\frac{3(g-1)}{k-1} \ge \frac{3g+12}{4}$, that gives $k \le 4$. This proves the last assertion of the theorem. 

Turning to the first assertion, let $X \subseteq \P^N$ be a smooth irreducible curve of genus $g \ge 2$ such that $T_X(k)$ is an Ulrich line bundle and assume that 
$$k \ge \frac{\sqrt{8g+1}-1}{2}.$$ 
Using Lemma \ref{posi}(i), the above inequality can be rephrased as 
\begin{equation}
\label{bd}
g \ge \frac{2}{9}d^2-d+1.
\end{equation}
Consider a general linear projection $X'$ of $X$ to $\P^3$. Note that $X' \cong X$, hence $T_{X'}(k)$ is Ulrich. We first observe that $X'$ cannot be a complete intersection (hence, in particular, $X'$ is nondegenerate), for otherwise $T_{X'}(k)=lH$ for some $l \in \Z$. Now $T_{X'}(k)$, being Ulrich, is globally generated by Lemma \ref{ulr}(vi), hence $l \ge 0$. Also $0=H^0(T_{X'}(k-1))=H^0((l-1)H)$ and therefore $l=0$. Hence Lemma \ref{ulr}(vii) gives that $d=h^0(T_{X'}(k))=1$, a contradiction.

Using Lemma \ref{posi}(i) and Castelnuovo's bound, we get that either $(d,g,k)=(6,3,2)$ or $d \ge 7$. 

Suppose that $d \ge 7$. 

We aim to show that $X'$ must lie on a smooth quadric. 

To this end, observe that \eqref{bd} and Lemma \ref{posi}(i) imply that 
\begin{equation}
\label{cub}
g >  \begin{cases} \frac{1}{6}d(d-3)+1 & {\rm if} \ d \equiv 0 \ ({\rm mod} \ 3) \\ \frac{1}{6}d(d-3)+\frac{1}{3} &{\rm if} \ d \equiv 1, 2 \ ({\rm mod} \ 3)
\end{cases}
\end{equation} 
unless $d=9$ and $g=10$. But in the latter case it is easy to show that if $X'$ does not lie on a quadric, then it is a complete intersection of two cubics, a contradiction. Therefore \eqref{cub} and \cite[Thm.~3.2]{ha2} give that $X'$ lies on a quadric $Q$. Moreover $Q$ is smooth, for otherwise it must be a cone, $d=2b+1$ is odd and $g=b^2-b$ by \cite[Ex.~V.2.9]{ha1}. But then Lemma \ref{posi}(i) gives that $4(k-1)=6b-9-\frac{3}{2b+1}$, and therefore $b=1$, a contradiction.

Thus $X'$ is a curve of type $(a,b)$ on $Q$, with $2 \le a \le b$. In particular $X'$ is linearly normal, hence $X=X'$. In the exact sequence
$$0 \to \O_Q(k+1-2a,k+1-2b) \to \O_Q(k+1-a,k+1-b) \to T_X(k-1) \to 0$$
since $H^0(T_X(k-1))=0$, we get that 
$$H^0(\O_Q(k+1-2a,k+1-2b))= H^0(\O_Q(k+1-a,k+1-b))$$ 
hence $k+1-b \le -1$, for otherwise $k+1-a \ge k+1-b \ge 0$, but then $X$ is a base-component of $|\O_Q(k+1-a,k+1-b)|$, contradicting the fact that this linear system is base-point-free. Therefore $b \ge k+2$. Moreover Lemma \ref{posi}(i) can be rewritten now as
$$(a+b)(k-1)=3((a-1)(b-1)-1)$$
that is
$$a = \frac{b(k+2)}{3b-k-2}$$
and it is readily seen that $b \ge k+2$ is equivalent to $a \le \frac{k}{2}+1$. Therefore $b \ge 2a$. But the maximum genus of a curve of type $(a, b)$ with $b \ge 2a$ and degree $d$ is attained when $b=\frac{2}{3}d$. Therefore
$$g \le (\frac{1}{3}d-1)(\frac{2}{3}d-1)=\frac{2}{9}d^2-d+1.$$
This shows that the inequality in \eqref{bd} cannot be strict, and therefore $g \le \frac{2}{9}d^2-d+1$, which is equivalent to \eqref{bdk}. Moreover, if equality holds in \eqref{bdk}, then it holds in \eqref{bd} and therefore $X$ is a curve of type $(a,b)$ with $b=\frac{2}{3}d$, hence $b=2a$ and $2a=b\ge k+2\ge 2a$, so that $k$ is even, $a= \frac{k}{2}+1$ and $b=k+2$. 

Next consider the only remaining case, $(d,g,k)=(6,3,2)$. 

Again $X'$ is linearly normal, hence $X=X'$. Also we have equality in \eqref{bd} and if $X$ lies on a quadric, then it must be of type $(2,4)$ and we are done in this case. Suppose therefore that $X$ does not lie on a quadric. Then it is easily seen that $\I_{X/\P^3}(3)$ is $0$-regular, hence globally generated, and we get that $X$ is contained in a smooth cubic. Therefore $X$ is one of the curves \eqref{cla} by Lemma \ref{632} and $T_X(2)$ is Ulrich.

Finally, to show that the bound \eqref{bdk} is sharp for every even $k \ge 0$, let $X$ be a curve of type $(\frac{k}{2}+1,k+2)$ on a smooth quadric $Q \subset \P^3$, so that $k = \frac{\sqrt{8g+1}-1}{2}$. It remains to show that $T_X(k)$ is Ulrich. Set $k=2c$. We have 
$$T_X(k-1)=-K_X+(k-1)H=\O_Q(c,-1)_{|X}$$ 
and the exact sequence
$$0 \to \O_Q(-1,-2c-3) \to \O_Q(c,-1) \to \O_Q(c,-1)_{|X} \to 0$$
shows that $H^i(\O_Q(c,-1)_{|X})=0$ for $i \ge 0$, since $H^i(\O_Q(c,-1))=H^i(\O_Q(-1,-2c-3))=0$ for $i \ge 0$. Hence $T_X(k)$ is Ulrich.
\end{proof}
\renewcommand{\proofname}{Proof}

\section{Surfaces}

Throughout this section we will have that $X \subseteq \P^N$ is a smooth irreducible surface.

We start by a characterization.

\begin{lemma}
\label{sup1}

Let $X \subseteq \P^N$ be a smooth irreducible surface. Then $T_X(k)$ is an Ulrich vector bundle if and only if
\begin{itemize}  
\item[(i)] $d= \frac{4(g-1)}{2k-1}$
\item[(ii)] $HK_X = \frac{(2k-3)d}{2}$.
\item[(iii)] $K_X^2 = 5 \chi(\O_X) + \frac{(k-1)(k-2)d}{2}$.
\item[(iv)] $H^0(T_X(k-1))=0$.
\item[(v)] $H^2(T_X(k-2))=0$.
\end{itemize} 
\end{lemma}
\begin{proof}
Note that (i) and (ii) are equivalent, since $HK_X=2(g-1)-d$. Now (ii) and (iii) are the conditions (2.2) in \cite[Prop.~2.1]{c}. Hence the lemma follows by loc. cit.
\end{proof}

Now we show the possible cases.

\begin{prop}
\label{sup}

Let $X \subseteq \P^N$ be a smooth irreducible surface. If $T_X(k)$ is an Ulrich vector bundle, the following hold:
\begin{itemize} 
\item[(i)] $0 \le k \le 3$. 
\end{itemize} 
Moreover, either
\begin{itemize} 
\item[(ii)] $k=0$ and $(X,H)=(\P^2,\O_{\P^2}(2))$, or
\item[(iii)] $k=1$ and $X$ is a Del Pezzo surface of degree $5$, or
\item[(iv)] $k=2, q=0$ and $X$ is a minimal surface of general type, or 
\item[(v)] $k=3$, $X$ is a minimal surface of general type with $2K_X \equiv 3H, K_X^2=\frac{9d}{4}, \chi(\O_X)=\frac{d}{4}$. Moreover $X$ is a ball quotient. 
\end{itemize} 
\end{prop}
\begin{proof}
We have that $k \ge 0$ by Lemma \ref{coh}(i). Now $H^1(T_X)=0$ by Lemma \ref{coh}(iv), that is $X$ is infinitesimally rigid and \cite[Thm.~1.3]{bc} implies that either $X$ is a minimal surface of general type or $X$ is a Del Pezzo surface of degree $j \ge 5$. In the latter case we have that $HK_X <0$ hence either $k=0$ and we get (ii) by Theorem \ref{k=0}, or $k=1$ and $K_X^2 = 5$ by Lemma \ref{sup1}(ii),(iii). This gives (iii). On the other hand, if $X$ is a minimal surface of general type then $HK_X > 0$, hence $k \ge 2$ by Lemma \ref{sup1}(ii). Next, the Hodge index theorem $H^2K_X^2 \le 
(HK_X)^2$ can be rewritten, using Lemma \ref{sup1}(ii),(iii) as 
$$\chi(\O_X) \le \frac{(2k^2-6k+5)d}{20}.$$ 
Similarly, the Bogomolov-Miyaoka-Yau inequality $K_X^2 \le 9 \chi(\O_X)$ can be rewritten as 
$$\chi(\O_X) \ge \frac{(k^2-3k+2)d}{8}.$$ 
Combining we get that 
$$\frac{(k^2-3k+2)d}{8} \le \frac{(2k^2-6k+5)d}{20}$$ 
and this gives that $k \le 3$ and moreover that, if $k=3$, then equality holds in both inequalities. Hence, when $k=3$ we have, as is well known, that $X$ is a ball quotient and that $H^2 K_X \equiv (HK_X)H$, that is $2K_X \equiv 3H$. Then $K_X^2=\frac{9d}{4}$ and $\chi(\O_X)=\frac{d}{4}$. Thus (i) and (v) are proved. Alternatively (i) follows by Theorem \ref{bou}. To see (iv) observe that since $k=2$ we have by the above that $X$ is a minimal surface of general type. Now if $p_g=0$ then $q=0$ by \cite[Lemma VI.1 and Prop.~X.1]{b2}. If $p_g \ne 0$ we have an inclusion $H^0(\Omega^1_X) \subseteq H^0(\Omega^1_X(K_X))$ hence $q = h^0(\Omega^1_X) \le h^0(\Omega^1_X(K_X))=h^2(T_X)=0$ since $T_X(2)$ is Ulrich. This proves (iv).
\end{proof}

We now characterize the case $k=1$ for surfaces.

\begin{thm} 
\label{k=1sup}

Let $X \subseteq \P^N$ be a smooth irreducible surface. Then $T_X(1)$ is an Ulrich vector bundle if and only if $X$ is a Del Pezzo surface of degree $5$ and $H=-2K_X$. Moreover in the latter case $T_X(1)$ is very ample.
\end{thm}
\begin{proof}
If $T_X(1)$ is an Ulrich vector bundle, then Proposition \ref{sup} implies that $X$ is a Del Pezzo surface of degree $5$ and $H^2+2HK_X = 0$. Let $\varepsilon : X \to \P^2$ be the blow-up map, with exceptional divisors $E_i$ over the points $P_i \in \P^2, 1 \le i \le 4$ and let $L$ be a line in $\P^2$. Then we can write
$$H \sim a\varepsilon^*L-\sum\limits_{i=1}^4 b_iE_i$$
and, as $H$ is very ample, we have, without loss of generality,
$$b_1 \ge b_2 \ge b_3 \ge b_4 \ge 1, a \ge b_1+ b_2 + 1$$ 
and $H^2+2HK_X = 0$ is
$$a^2-6a+4=\sum\limits_{i=1}^4 (b_i-1)^2.$$
Setting $c_i = b_i-1$, we get by Lemma \ref{conto} the following possibilities:
$$(a;b_1,b_2,b_3,b_4) \in \{(6;3,1,1,1), (6;2,2,2,2), (7;4,2,2,1), (9;4,4,4,3)\}.$$
In the case $(6;2,2,2,2)$ we have that $H=-2K_X$. We now exclude the other cases.

Let $H=6\varepsilon^*L-3E_1-E_2-E_3-E_4$. We will prove that $h^2(T_X(-1))=h^0(\Omega^1_X(H+K_X)) \ne 0$, so that $T_X(1)$ cannot be an Ulrich vector bundle. To this end observe that, since $\varepsilon^*\Omega^1_{\P^2} \subset \Omega^1_X$, we will be done in this case if we prove that 
$$H^0(\varepsilon^*\Omega^1_{\P^2} (H+K_X)) \ne 0.$$
Now $H+K_X=3\varepsilon^*L-2E_1$, hence 
$$H^0(\varepsilon^*\Omega^1_{\P^2} (H+K_X)) \cong H^0(\mathcal I_Z \otimes \Omega^1_{\P^2} (3))$$
where $Z \subset \P^2$ is the $0$-dimensional subscheme of length $2$ supported on $P_1$.
Finally  
$$h^0(\mathcal I_Z \otimes \Omega^1_{\P^2} (3)) \ge h^0(\Omega^1_{\P^2} (3)) - 6 = 2 > 0$$
and we are done in this case. 

Consider now the exact sequences, for any $1 \le i \le 4$,
$$0 \to \O_{E_i}(-E_i) \to {\Omega^1_X}_{|E_i} \to \Omega^1_{E_i} \to 0$$
that is
$$0 \to \O_{\P^1}(1) \to {\Omega^1_X}_{|E_i} \to \O_{\P^1}(-2) \to 0$$
from which we get, for any $1 \le i \le 4$, that
\begin{equation}
\label{uno}
h^1({\Omega^1_X}_{|E_i})=1
\end{equation}
and 
\begin{equation}
\label{due}
H^1({\Omega^1_X}_{|E_i} \otimes \O_{\P^1}(2)) = 0. 
\end{equation}
In the two remaining cases we will prove that $h^1(T_X(-1))=h^1(\Omega^1_X(H+K_X)) \ne 0$.

Let $H=7\varepsilon^*L-4E_1-2E_2-2E_3-E_4$.

Note that $H+K_X-E_4+E_1=4\varepsilon^*L-2E_1-E_2-E_3-E_4$ is very ample by \cite[Cor.~4.6]{d}, hence
$$H^2(\Omega^1_X(H+K_X-E_4+E_1))=0$$ 
by Bott vanishing \cite[Thm.~2.1]{t}. Then the exact sequence
$$0 \to \Omega^1_X(H+K_X-E_4) \to \Omega^1_X(H+K_X-E_4+E_1) \to {\Omega^1_X}_{|E_1}(H+K_X-E_4+E_1) \to 0$$
and \eqref{due} imply that $H^2(\Omega^1_X(H+K_X-E_4))=0$. Now the exact sequence
$$0 \to \Omega^1_X(H+K_X-E_4) \to \Omega^1_X(H+K_X) \to {\Omega^1_X}_{|E_4}(H+K_X) \to 0$$
and \eqref{uno} imply that
$$h^1(\Omega^1_X(H+K_X)) \ge h^1({\Omega^1_X}_{|E_4}(H+K_X))=h^1({\Omega^1_X}_{|E_4})=1.$$

Let $H=9\varepsilon^*L-4E_1-4E_2-4E_3-3E_4$. 

Let $C \in |\varepsilon^*L-E_2-E_3|$ be the strict transform of a line through $P_2$ and $P_3$. Note that $H+K_X-C+E_1=5\varepsilon^*L-2E_1-2E_2-2E_3-2E_4$ is very ample by \cite[Cor.~4.6]{d}, hence
$$H^2(\Omega^1_X(H+K_X-C+E_1))=0$$ 
by Bott vanishing \cite[Thm.~2.1]{t}. Then the exact sequence
$$0 \to \Omega^1_X(H+K_X-C) \to \Omega^1_X(H+K_X-C+E_1) \to {\Omega^1_X}_{|E_1}(H+K_X-C+E_1) \to 0$$
and \eqref{due} imply that $H^2(\Omega^1_X(H+K_X-C))=0$. Now the exact sequence
$$0 \to \O_{C}(-C) \to {\Omega^1_X}_{|C} \to \Omega^1_C \to 0$$
that is
$$0 \to \O_{\P^1}(1) \to {\Omega^1_X}_{|C} \to \O_{\P^1}(-2) \to 0$$
gives that $h^1({\Omega^1_X}_{|C}) = 1$. Finally, from the exact sequence 
$$0 \to \Omega^1_X(H+K_X-C) \to \Omega^1_X(H+K_X) \to {\Omega^1_X}_{|C}(H+K_X) \to 0$$
using that $(H+K_X)C = 0$, we get that
$$h^1(\Omega^1_X(H+K_X)) \ge h^1({\Omega^1_X}_{|C}(H+K_X)) = h^1({\Omega^1_X}_{|C}) =1.$$
This completes the proof under the assumption that $T_X(1)$ is an Ulrich vector bundle.

Suppose now that $X$ is a Del Pezzo surface of degree $5$ and $H=-2K_X$. Setting $k=1$ in Lemma \ref{sup1}, we have that $d=4(g-1)$ and, in order to verify that $T_X(1)$ is an Ulrich vector bundle, we need to check that $H^0(T_X)=H^2(T_X(-1))=0$. The first vanishing is well known. As for the second, we first observe that for $i < 2$ we have
$$h^i(T_X(-1))=h^{2-i}(\Omega^1_X(H+K_X))=h^{2-i}(\Omega^1_X(-K_X))=0$$
by Bott vanishing \cite[Thm.~2.1]{t}. Therefore $h^2(T_X(-1))=\chi(T_X(-1))=d-4(g-1)=0$. Finally, as $X$ does not contain lines in the embedding given by $H=-2K_X$, we have that $T_X(1)$ is very ample by \cite[Thm. 1]{ls}
\end{proof}

\section{Properties of complete intersections}

We collect some properties inherited by the complete intersections $X_i$ of $X$ (as in Notation \ref{nota}), when $T_X(k)$ is an Ulrich vector bundle.

\begin{lemma}
\label{q}
Let $X \subseteq \P^N$ be a smooth irreducible variety of dimension $n \ge 1$. Then $q(X)=q(X_i)$ for $2 \le i \le n$.
\end{lemma}
\begin{proof} 
By Kodaira vanishing we have that $H^1(\O_{X_{i+1}}(-1))=H^2(\O_{X_{i+1}}(-1))=0$ as long as $2 \le i \le n-1$. Then the exact sequences
$$0 \to \O_{X_{i+1}}(-1) \to \O_{X_{i+1}} \to \O_{X_i} \to 0$$
imply that $h^1(\O_{X_{i+1}})=h^1(\O_{X_i})$ for every $2 \le i \le n-1$, hence $q(X)=q(X_i)$ for $2 \le i \le n$.
\end{proof}

\begin{lemma}
\label{jj}
Let $X \subseteq \P^N$ be a smooth irreducible variety of dimension $n \ge 1$. Suppose that $k \le n-2$ and that $T_X(k)$ is Ulrich. Then $H^i(\O_{X_i})=0$ for all $i$ such that $\max\{1,k+1\} \le i \le n-1$.
\end{lemma}
\begin{proof}
Assume that $\max\{1,k+1\} \le i \le n-1$. Since ${T_X}_{|{X_i}}(k)$ is Ulrich, it follows by Lemma \ref{ulr}(iv) that $H^i({T_X}_{|{X_i}}(k+m))=0$ for all $m \ge -i$, hence $H^i({T_X}_{|{X_i}}(-1))=0$. Now the exact sequence 
$$0 \to T_{X_i}(-1) \to {T_X}_{|{X_i}}(-1) \to \O_{X_i}^{\oplus (n-i)} \to 0$$ 
implies that $H^i(\O_{X_i})=0$.
\end{proof}

\begin{lemma}
\label{11}
Let $X \subseteq \P^N$ be a smooth irreducible variety of dimension $n \ge 1$. Suppose that $k \le n-1$ and that $T_X(k)$ is Ulrich. Assume that $H^i(\O_X)=0$ for all $i \ge 1$. Then $H^i(\O_{X_j})=0$ for all $i \ge 1$ and for all $j$ such that $\max\{1,k+1\} \le j \le n$.
\end{lemma}
\begin{proof}
Assume that $i \ge 1$ and $\max\{1,k+1\} \le j \le n$. We prove the lemma by induction on $n-j \ge 0$. 

If $n-j=0$ then $X_j=X_n=X$ and $H^i(\O_X)=0$ for all $i \ge 1$ just by our assumption. 

Next suppose that $n-j \ge 1$, so that $\max\{1,k+1\} \le j \le n-1$, hence, in particular $k \le n-2$. Consider the exact sequence 
$$0 \to \O_{X_{j+1}}(-1) \to \O_{X_{j+1}} \to \O_{X_j} \to 0.$$ 
If $j=i$, we have that $H^j(\O_{X_j})=0$ by Lemma \ref{jj}. Also, we have by induction that $H^i(\O_{X_{j+1}})=0$. Now $H^{i+1}(\O_{X_{j+1}}(-1) )=0$ by Kodaira vanishing if $i+1<j+1$ and by dimension reasons if $i+1>j+1$. Thus $H^i(\O_{X_j})=0$ if $i \ne j$ and we are done.
\end{proof}

We now collect some properties of the $X_i$'s that hold when $T_X(1)$ is Ulrich.

\begin{lemma}
\label{sezioni} 
Let $X \subseteq \P^N$ be a smooth irreducible variety of dimension $n \ge 2$. Suppose that $T_X(1)$ is an Ulrich vector bundle. Then:
\begin{itemize}
\item[(i)] $H^1(\O_{X_i})=0$ for $2 \le i \le n$.
\item[(ii)] $H^2(\O_{X_i}) = 0$ for $1 \le i \le n$.
\item[(iii)] $H^1(\O_{X_i}(1)) = 0$ for $1 \le i \le n$.
\item[(iv)] $H^2(\O_{X_i}(1)) = 0$ for $1 \le i \le n$. 
\item[(v)] $h^0(\O_{X_i}(1)) = d-g+i$ for $1 \le i \le n$.
\item[(vi)] $d \ge n+3$.
\end{itemize}
\end{lemma}
\begin{proof} 
We have $H^i(\O_X)=0$ for $i \ge 1$ by Lemma \ref{posi}(iii). Now (i) follows by Lemma \ref{q} and (ii) follows by Lemma \ref{11}. To see (iii) observe that, if $i=1$ we have that $X_1=C$ and $T_X(1)_{|C}$ is an Ulrich vector bundle on $C$ by Lemma \ref{ulr}(ix), hence $H^1({T_X}_{|C})=0$. Then the exact sequence
$$0 \to T_C \to {T_X}_{|C} \to \O_C(1)^{\oplus (n-1)} \to 0$$
shows that $H^1(\O_C(1))=0$. If $i \ge 2$, since $H^1(\O_{X_i})=0$ by (i), the exact sequences
\begin{equation}
\label{sezioni2}
0 \to \O_{X_i} \to \O_{X_i}(1) \to \O_{X_{i-1}}(1) \to 0
\end{equation}
imply by induction that $H^1(\O_{X_i}(1)) = 0$ and we get (iii). Now (iv) is obvious for $i=1$, while, for $i \ge 2$, the exact sequences \eqref{sezioni2} and (ii) show by induction that $H^2(\O_{X_i}(1)) = 0$. This proves (iv). Note that (v) follows for $i=1$ by Riemann-Roch and (iii). For $i \ge 2$, the exact sequences \eqref{sezioni2} and (i) show by induction that $h^0(\O_{X_i}(1)) = 1 + h^0(\O_{X_{i-1}}(1))=d-g+i$, that is (v). Finally, to see (vi), observe that $g-1=\frac{n-1}{n+2}d$ by Lemma \ref{posi}(i), hence $g \ge 2$ and (v) gives that $\frac{3d}{n+2} = h^0(\O_C(1)) \ge 3$, so that $d \ge n+2$. Moreover, if equality holds, we get that $g=n$ and $h^0(\O_X(1))=n+2$ by (v), hence $X \subset \P H^0(H) = \P^{n+1}$ is a hypersurface of degree $n+2$, so that $K_X=0$, contradicting Lemma \ref{posi}(ii). Hence (vi) is proved.
\end{proof}

\section{$T_X(k)$ Ulrich and special varieties in adjunction theory}
\label{sez7}

In this section we exclude some special varieties frequently arising in adjunction theory, under the hypothesis that $T_X(k)$ is an Ulrich vector bundle.
The cases $(X,H)=(\P^n,\O_{\P^n}(1)), (Q_n,\O_{Q_n}(1))$ have been already treated in Lemmas \ref{proj} and \ref{quad}.

We start by recalling the following (see \cite{bs, i}).

\begin{defi}
Let $E$ be an effective divisor on $(X,H)$. The divisor $E$ is called {\it exceptional}
\begin{itemize}
\item[(i)] {\it of type 1} if $(E,H_{|E})\cong (\P^{n-1},\O_{\P^{n-1}}(1))$ and $N_{E/X} \cong \O_{\P^{n-1}}(-1)$,
\item[(ii)] {\it of type 2} if $(E,H_{|E})\cong (\P^{n-1},\O_{\P^{n-1}}(1))$ and $N_{E/X} \cong \O_{\P^{n-1}}(-2)$, 
\item[(iii)] {\it of type 3} if $(E,H_{|E})\cong (Q_{n-1},\O_{Q_{n-1}}(1))$ and $N_{E/X} \cong \O_{Q_{n-1}}(-1)$,
\item[(iv)] {\it of type 4} if $(E,H_{|E})$ is a linear $\P^{n-2}$-bundle over a smooth curve $B$ and $(N_{E/X})_{|F} \cong \O_{\P^{n-2}}(-1)$, where $F$ is a fiber of the structure morphism $E \to B$. 
\end{itemize}
\end{defi}

Often these exceptional divisors will not be present under the condition that $T_X(k)$ is Ulrich. To see this we first prove

\begin{lemma}
\label{nogg} 
Let $W$ be a variety of dimension $s \ge 1$ and let $\O_W(1)$ be a very ample line bundle. Then $\Omega^1_W(1)$ is not globally generated if:
\begin{itemize}
\item[(i)] $(W,\O_W(1)) \cong(\P^s,\O_{\P^s}(1))$.
\item[(ii)]  $(W,\O_W(1))$ is a (possibly singular) quadric hypersurface in $\P^{s+1}$ and $s \ge 2$.
\item[(iii)] $(W,\O_W(1))$ is a smooth Del Pezzo variety, $s \ge 2$ and $(W,\O_W(1)) \not\in \{(\P^2,\O_{\P^2}(3)), (Q_2,\O_{Q_2}(2)),$ $(\P^3,\O_{\P^3}(2))\}$.    
\end{itemize}
\end{lemma}
\begin{proof} 
(i) follows from $\det (\Omega^1_{\P^s}(1)) =\O_{\P^s}(-1)$. To see (ii), observe that the restricted Euler sequence
$$0 \to {\Omega^1_{\P^{s+1}}}_{|W}(1) \to H^0(\O_W(1)) \otimes \O_W \to \O_W(1) \to 0$$
implies that $H^0({\Omega^1_{\P^{s+1}}}_{|W}(1))=0$. Now the exact sequence
$$0 \to \O_{\P^{s+1}}(-3) \to \O_{\P^{s+1}}(-1) \to \O_W(-1) \to 0$$
implies that $H^1(\O_W(-1))=0$ and the dual normal bundle sequence
$$0 \to \O_W(-1) \to {\Omega^1_{\P^{s+1}}}_{|W}(1) \to \Omega^1_W(1) \to 0$$
gives that $H^0(\Omega^1_W(1))=0$, hence $\Omega^1_W(1)$ is not globally generated. Next, to see (iii), observe that, from the classification of Del Pezzo varieties \cite[Thm.~3.3.1]{ip} it follows, for the surface section $W_2$, that $(W_2,\O_{W_2}(1)) \not\in \{(\P^2,\O_{\P^2}(3)), (Q_2,\O_{Q_2}(2))$. Hence, as is well known, $W_2$, and hence $W$, contains a line $L$. But now the surjection $\Omega^1_W(1) \to \Omega^1_L(1)=\O_{\P^1}(-1)$ gives that $\Omega^1_W(1)$ is not globally generated. 
\end{proof}

Now 

\begin{lemma}
\label{noexc}
Let $X \subseteq \P^N$ be a smooth irreducible variety of dimension $n \ge 1$. Assume that $T_X(k)$ is an Ulrich vector bundle. We have: 
\begin{itemize}
\item[(i)] If $k \ge 1$ and $n \ge 2$, then $(X,H)$ does not contain any exceptional divisor of type 1.
\item[(ii)]  If $k \ge 2$ and $n \ge 2$, then $(X,H)$ does not contain any exceptional divisor of type 2.
\item[(iii)] If $k \ge 2$ and $n \ge 3$, then $(X,H)$ does not contain any exceptional divisors of types 3 or 4.
\end{itemize}
\end{lemma}
\begin{proof} 
Let $E$ be an exceptional divisor. It follows from Lemma \ref{gg} that $\Omega^1_E({K_X}_{|E}+(n+1-k)H_{|E})$ is globally generated. Now
$$\Omega^1_E({K_X}_{|E}+(n+1-k)H_{|E}) \cong \begin{cases} 
\Omega^1_{\P^{n-1}}(2-k) & \textrm{if $E$ is of type 1;} \\
\Omega^1_{\P^{n-1}}(3-k) & \textrm{if $E$ is of type 2;} \\
\Omega^1_{Q_{n-1}}(3-k) & \textrm{if $E$ is of type 3.}
\end{cases}$$
Further, when $E$ is of type 4, let $F$ be a fiber of the structure morphism of $E$. Again it follows from Lemma \ref{gg} that $\Omega^1_F({K_X}_{|F}+(n+1-k)H_{|F}) \cong \Omega^1_{\P^{n-2}}(3-k)$ is globally generated. Consequently, we draw the conclusions from Lemma \ref{nogg}.
\end{proof}

We now recall

\begin{defi}
\label{special}

We say that $(X,H)$ is a {\it linear $\P^k$-bundle} over a smooth variety $B$ if $(X, H) \cong (\P(\F), \O_{\P(\F)}(1))$, where $\F$ is a very ample vector bundle on $B$ of rank $k+1$.

We say that $(X,H)$ as above is a {\it scroll (respectively a quadric fibration; respectively a Del Pezzo fibration) over a normal variety $Y$ of dimension $m$} if there exists a surjective morphism with connected fibers  $\phi: X \to Y$ such that $K_X+(n-m+1)H=\phi^*\L$ (respectively $K_X+(n-m)H=\phi^*\L$; respectively $K_X+(n-m-1)H=\phi^*\L$), with $\L$ ample on $Y$.
\end{defi}

We now use the fibration to exclude several varieties as above, when $T_X(k)$ is an Ulrich vector bundle.

\begin{lemma}
\label{noadj}
Let $X \subseteq \P^N$ be a smooth irreducible variety of dimension $n \ge 1$. Assume that $T_X(k)$ is an Ulrich vector bundle. Let $f : X \to B$ be a fibration onto a normal variety $B$ of dimension $m \ge 1$, with general fiber $F$. Then:
\begin{itemize}
\item[(i)] If $m \le \min\{n-1, k+1\}$, then $(F,H_{|F}) \ne (\P^{n-m},\O_{\P^{n-m}}(1))$.
\item[(ii)] If $m \le \min\{n-2, k\}$, then $(F,H_{|F}) \ne (Q_{n-m},\O_{Q_{n-m}}(1))$.  
\item[(iii)] if $m \le \min\{n-2, k-1\}$, then $(F,H_{|F})$ is not a Del Pezzo variety, unless $(F,H_{|F}) \in \{(\P^2,\O_{\P^2}(3)), (Q_2,\O_{Q_2}(2)), (\P^3,\O_{\P^3}(2))\}$.
\end{itemize}
\end{lemma}
\begin{proof}
We have that 
$$\Omega^1_F(K_F+(n+1-k)H_{|F}) \cong \begin{cases} 
\Omega^1_{\P^{n-m}}(m-k) & \textrm{if $(F,H_{|F}) = (\P^{n-m},\O_{\P^{n-m}}(1))$;} \\
\Omega^1_{Q_{n-m}}(m-k+1) & \textrm{if $(F,H_{|F}) = (Q_{n-m},\O_{Q_{n-m}}(1))$;} \\
\Omega^1_{F}(m-k+2) & \textrm{if $(F,H_{|F})$ is a Del Pezzo variety.}
\end{cases}$$
Now $\Omega^1_F(K_F+(n+1-k)H_{|F})$ is globally generated by Lemma \ref{gg}. Hence, Lemma \ref{nogg} gives that, in each of the three cases, the inequality in $m, n, k$ is not satisfied.
\end{proof}

We get a very useful consequence.

\begin{lemma}
\label{posi2} 

Let $X \subseteq \P^N$ be a smooth irreducible variety of dimension $n \ge 2$. If $T_X(k)$ is an Ulrich vector bundle, then $K_X+(n-1)H$ is nef and $H^0(K_X+(n-1)H) \ne 0$, unless $(X,H,k)=(\P^2,\O_{\P^2}(2),0)$ (the latter case actually occurs, see Theorem \ref{k=0}).
\end{lemma}
\begin{proof}
Recall that $H^0(K_X+(n-1)H) \ne 0$ if and only if $K_X+(n-1)H$ is nef by \cite[Cor.~7.2.8]{bs}. Now if $K_X+(n-1)H$ is not nef, it follows by \cite[Prop.'s ~7.2.2, 7.2.3 and 7.2.4]{bs} that $(X,H)$ is either $(\P^n,\O_{\P^n}(1)), (Q_n,\O_{Q_n}(1))$, a linear $\P^{n-1}$-bundle over a smooth curve or $(\P^2,\O_{\P^2}(2))$. The first three cases are excluded by Lemmas \ref{proj}, \ref{quad} and \ref{noadj}(i), while in the fourth case we have $g=0$, hence $k=0$ by Lemma \ref{posi}(i). 
\end{proof}

We can now prove Theorem \ref{k=0}.

\renewcommand{\proofname}{Proof of Theorem \ref{k=0}}
\begin{proof} The assert is clear if either $(X,H)=(\P^1,\O_{\P^1}(3))$ or $(\P^2,\O_{\P^2}(2))$. Vice versa assume that $T_X$ is Ulrich for $H$. If $n=1$, since $T_X=-K_X$ is globally generated by Lemma \ref{ulr}(vi), we have that $X$ is either $\P^1$ or an elliptic curve. Now the latter is excluded by Lemma \ref{posi}(i), while in the former case $T_X=\O_{\P^1}(2)$ Ulrich implies that $H=\O_{\P^1}(3)$. Now assume that $n \ge 2$. Then Lemma \ref{posi2} gives that either $(X,H)=(\P^2,\O_{\P^2}(2))$, or $K_X+(n-1)H$ is nef, leading, by Lemma  \ref{posi}(ii), to the contradiction 
$$0 \le (K_X+(n-1)H)H^{n-1}=-\frac{2d}{n+2}.$$ 
\end{proof}
\renewcommand{\proofname}{Proof}

The following result will also be useful.

\begin{lemma}
\label{notbundle}
Let $X \subseteq \P^N$ be a smooth irreducible variety of dimension $n \ge 2$. Suppose that $k \ge 1$ and that $T_X(k)$ is an Ulrich vector bundle. Assume that $X \cong \P(\F)$ is a projective bundle over a normal projective variety $B$ of dimension $1 \le m \le n-1$. Then $B$ is smooth and $\F$ is simple. In particular, if $m=1$, then $q(X) \ne 0$.
\end{lemma}
\begin{proof} 
Let $\pi : X \cong \P(\F) \to B$ be the structure morphism and let $\xi$ be the tautological bundle of $\P(\F)$. By twisting $\F$ with a sufficiently ample line bundle we can assume that $\xi$ is ample. Then \cite[Prop.~3.2.1]{bs} implies that $B$ is smooth. Since $H^0(T_X)=0$, the cohomology of the exact sequence
$$0 \to T_{X/B} \to T_X \to \pi^*T_B \to 0$$
gives that $H^0(T_{X/B})=0$. Now the cohomology of the exact sequence
$$0 \to \O_X \to \pi^*\F^* \otimes \xi \to T_{X/B} \to 0$$
implies that
$$h^0(\F \otimes \F^*)= h^0(\pi^*\F^* \otimes \xi)=h^0(\O_X)=1.$$
Now if $m=1$ and $q(X)=0$ we have that $B \cong \P^1$, hence $\F$ cannot be simple since $\rk \F = n \ge 2$. 
\end{proof}

Next we prove three results for $k=2$. 

For the first one, in order to apply the results of T. Fujita in \cite{f1}, we give the following definition, that coincides with the one in \cite{f1} when $B$ is smooth.

\begin{defi}
\label{minimal}
Let $f: X \to B$ be a fibration over a curve, $L$ an ample line bundle on $X$ such that on the general fiber $F$ we have that $K_F = -(n-2)L_{|F}$. We say that $f$ is {\it minimal} if there is a line bundle $\L$ on $B$ such that $K_X+(n-2)L=f^*\L$.
\end{defi}

Then we have

\begin{lemma}
\label{DP4}
Let $X \subseteq \P^N$ be a smooth irreducible variety of dimension $n \ge 4$. Suppose that $k \ge 2$ and that $k=2$ if $n=4$. Moreover assume that $T_X(k)$ is an Ulrich vector bundle. Then:
\begin{itemize}
\item [(i)] $(X,H)$ is not a Del Pezzo fibration over a smooth curve.
\item [(ii)] If $n=4$ and $K_X+2H$ is ample, then $(X,K_X+2H)$ is not a minimal $(\P^3,\O_{\P^3}(2))$-fibration over a smooth curve.
\end{itemize}
\end{lemma}
\begin{proof} 
For the sake of contradiction, let $L$ be $H$ in case (i) and $K_X+2H$ in case (ii). Assume that we have a fibration $f: X \to B$ over a smooth curve $B$ such that $(X,L)$ is a Del Pezzo fibration in case (i) (see Definition \ref{special}) and $(X,L)$ is a minimal $(\P^3,\O_{\P^3}(2))$-fibration in case (ii). Note that $f$ is minimal also in case (i) by Definition \ref{special}.

Let $F$ be a general fiber of $f$. In case (i) we have that $F$ is a smooth variety of dimension $n-1$ and $K_F = {K_X}_{|F}=-(n-2)L_{|F}$, hence $F$ is a Del Pezzo variety. Since $L=H$, Lemma \ref{noadj} implies that $(F,H_{|F}) = (\P^3,\O_{\P^3}(2))$, hence $n=4$. Thus $(F,L_{|F})$ is the same in both cases.

We now claim that every fiber of $f$ 
is irreducible. 
Indeed, if not, let $F_0$ be a reducible fiber. Since $F_0$ is connected, it must be singular, hence we can apply \cite[Table (2.20)]{f1}. It follows that we are in case (2.17) of \cite[Table (2.20)]{f1}, the degree of $F_0$ is $8$ and, if $D$ is an irreducible component of $F_0$, then $(D,L_{|D})$ is a scroll over $\P^1$ and ${K_X}_{|D}=-2L_{|D}$. Denoting a fiber of the structure morphism $D \to \P^1$ by $F' \cong \P^2$, we obtain $L_{|F'} = \O_{\P^2}(1)$ and ${K_X}_{|F'} = -2L_{|F'}=\O_{\P^2}(-2)$. Set $H_{|F'} = \O_{\P^2}(a)$. In case (ii) we have that 
$$\O_{\P^2}(1) = L_{|F'} = (K_X+2H)_{|F'} = \O_{\P^2}(-2+2a)$$
a contradiction. In case (i) we have that $L=H$ and $a=1$. But now Lemma \ref{gg} gives that $\Omega^1_{\P^2}({K_X}_{|F'}+3H_{|F'}) \cong \Omega^1_{\P^2}(1)$ is globally generated, contradicting Lemma \ref{nogg}(i). Thus every fiber of $f$ is irreducible.

Now \cite[(4.8)]{f1} implies that every fiber of $f$ is $\P^3$. Since $B$ is a smooth curve, it follows, as is well known, that $X$ is a projective bundle over $B$. On the other hand, since $n=4$, we have that $k=2$ and $q(X)=0$ by Lemma \ref{posi}(iii), contradicting Lemma \ref{notbundle}.
\end{proof}

\begin{lemma}
\label{noqf4}
Let $X \subseteq \P^N$ be a smooth irreducible variety of dimension $4$. Suppose that $K_X+2H$ is ample and that $T_X(2)$ is an Ulrich vector bundle. Then $(X,K_X+2H)$ is not a quadric fibration over a smooth curve.
\end{lemma}
\begin{proof} 
Suppose that $(X,K_X+2H)$ is a quadric fibration $\pi : X \to B$ over a smooth curve. Since $\chi(\O_X)=1$ by Lemma \ref{posi}(iii), it follows from \cite[Sect.~1.1, eq.~(8)]{la} that $B \cong \P^1$. Moreover, \cite[Sect.~0.1]{la} gives that if $\pi_*(K_X+2H) \cong \bigoplus\limits_{i=0}^4 \O_{\P^1}(a_i)$ and $e =\sum_{i=0}^4 a_i$, then there is $b \in \Z$ such that
$$(K_X+2H)^4=2e-b$$
by \cite[Sect.~1.1, eq.~(3)]{la} and
$$K_X^i(K_X+2H)^{4-i}=(-3)^i2e+(-3)^{i-1}(-4i+2ie+(3-2i)b) \ \hbox{for} \ 1 \le i \le 4$$
by \cite[Sect.~1.1, eq.~(4)]{la}. Solving these five equations we obtain 
$$K_XH^3=4e-28b-104 \ \hbox{and} \ d=H^4=16b+64$$
and therefore Lemma \ref{posi}(iii) gives
\begin{equation}
\label{e}
13d=48(2+e).
\end{equation}
Since $T_X(2)$ is Ulrich we have that $H^4(T_X(-2))=0$ and the exact sequence
$$0 \to T_{X/\P^1}(-2) \to T_X(-2) \to (\pi^*T_{\P^1})(-2) \to 0$$
implies that $H^4((\pi^*T_{\P^1})(-2))=0$. Hence, by Serre duality
$$0= h^4((\pi^*T_{\P^1})(-2))=h^0((\pi^*\O_{\P^1}(-2))(K_X+2H)) = h^0(\pi_*(K_X+2H) \otimes \O_{\P^1}(-2))=\sum\limits_{i=0}^4 h^0(\O_{\P^1}(a_i-2))$$
and therefore $a_i \le 1$ for $0 \le i \le 4$. But then $e \le 5$ and \eqref{e} gives that $1 \le e+2 \le 7$ is divisible by $13$, a contradiction.
\end{proof}

\begin{lemma}
\label{nosc4}
Let $X \subseteq \P^N$ be a smooth irreducible variety of dimension $4$. Suppose that $K_X+2H$ is ample and that $T_X(2)$ is an Ulrich vector bundle. Then $(X,K_X+2H)$ is not a linear $\P^2$-bundle over a smooth surface.
\end{lemma}
\begin{proof} 
Assume by contradiction that we have a $\P^2$-bundle structure $\pi: X \cong \P(\F) \to B$ onto a smooth surface $B$, with $K_X+2H = \xi$, the tautological bundle, where $\F$ is a rank $3$ vector bundle on $B$. Then $H = a\xi-\pi^*M$ for some $a \in \Z$ and $M \in \Pic(B)$, so that
$$\xi = K_X+2H = (2a-3)\xi + \pi^*(K_B+c_1(\F)-2M)$$
giving $a=2$ and $2M=K_B+c_1(\F)$, thus
\begin{equation}
\label{h}
H \equiv 2\xi-\frac{1}{2}\pi^*(K_B+c_1(\F)).
\end{equation}
We will also use Grothendieck's relation $\sum\limits_{j=0}^3 (-1)^j \xi^{3-j}\pi^*c_j(\F)=0$, that is
\begin{equation}
\label{gr}
\xi^3=\xi^2 \pi^*c_1(\F)-\xi \pi^*c_2(\F).
\end{equation}
Since $\xi^2f=1$ for every fiber $f$ of $\pi$, we get from \eqref{gr} that 
\begin{equation}
\label{gr2}
\xi^3\pi^*c_1(\F)=c_1(\F)^2, \xi^3\pi^*K_B=K_Bc_1(F) \ \hbox{and} \ \xi^4=c_1(\F)^2-c_2(\F).
\end{equation}
We first collect some invariants of $X$ and $B$.

\begin{claim}
\label{XeB}
We have:
\begin{itemize} 
\item[(i)] $K_XH^3=-\frac{2}{3}d$.
\item[(ii)] $\chi(\O_X)=1$.
\item[(iii)] $\chi(\O_X(H))=2+\chi(\O_S)-\frac{d}{6}$.
\item[(iv)] $h^0(K_X+2H)=\chi(\O_S)-1$. 
\item[(v)] $\chi(\O_B)=1$. 
\end{itemize} 
\end{claim}
\begin{proof}
(i) is obtained by Lemma \ref{posi}(ii). Now Lemma \ref{posi}(iii) gives that $H^i(\O_X)=0$ for $i \ge 1$, hence $H^i(\O_B)=0$ for $i \ge 1$, giving (ii) and (v). Next, to see (iii), consider the exact sequences 
$$0 \to \O_{X_i} \to \O_{X_i}(H) \to \O_{X_{i-1}}(H) \to 0$$
for $i=4, 3$. They give $\chi(\O_X(H))=1+\chi(\O_{X_3}(H))=2+\chi(\O_S(H))$ and (iii) follows by Riemann-Roch since $H_{|S}^2=d$ and $H_{|S}K_S=(K_X+2H)H^3=\frac{4}{3}d$ by (i). To see (iv), observe that, since $R^j \pi_*(-\xi)=0$ for every $j \ge 0$, we have that $H^i(K_X+H)=H^i(-\xi+\pi^*(K_B+c_1(\F)-M))=0$ for every $i$. Hence the exact sequence
$$0 \to K_X+H \to K_X+2H \to K_{X_3}+H_{|X_3} \to 0$$
implies that 
\begin{equation}
\label{una}
h^0(K_X+2H)=h^0(K_{X_3}+H_{|X_3}).
\end{equation} 
Now we have $q(S)=0$ by Lemma \ref{q} and $H^i(K_{X_3})=0$ for $i=0,1$ by Serre duality and Lemma \ref{11}. Hence the exact sequence 
$$0 \to K_{X_3} \to K_{X_3}+H_{|X_3} \to K_S \to 0$$
shows that 
$$\chi(\O_S)-1=p_g(S)=h^0(K_S)=h^0(K_{X_3}+H_{|X_3})$$ 
and we get (iv) by \eqref{una}.
\end{proof}
We continue the proof of the lemma. 

Next, we collect some relations among the invariants related to $\xi, K_B$ and the Chern classes of $\F$.

\begin{claim}
\label{gro2}
The following identities hold:
\begin{itemize} 
\item[(i)] $d-6c_1(\F)^2+16c_2(\F)-6K_B^2+4K_Bc_1(\F)=0$.
\item[(ii)] $K_Bc_1(\F)-8+2\chi(\O_S)-c_1(\F)^2+2c_2(\F)=0$.
\item[(iii)] $K_B^2-1+c_1(\F)^2-3c_2(\F)=0$.
\item[(iv)] $3K_B^2+c_1(\F)^2-2c_2(\F)-30=0$.
\item[(v)] $4-K_Bc_1(\F)+\frac{9}{4}K_B^2 + \frac{7}{4} c_1(\F)^2 - 5c_2(\F)-\chi(\O_S)+\frac{1}{6}d=0$.
\end{itemize} 
\end{claim}
\begin{proof}
We have by \eqref{h} and \eqref{gr2} that 
$$d = H^4 = (2\xi-\frac{1}{2}\pi^*(K_B+c_1(\F)))^4=16(c_1(\F)^2-c_2(\F))+6K_B^2-4K_Bc_1(\F)-10c_1(\F)^2$$ 
that is (i). To see (ii) observe that, since $\pi_*\xi = \F$ and $R^j\pi_*\xi=0$ for $j>0$ we have $H^i(\F)=H^i(\xi)=H^i(K_X+2H)=0$ for $i > 0$ by Kodaira vanishing. Also, by Claim \ref{XeB}(iv), (v) and Riemann-Roch we get
$$\chi(\O_S)-1=h^0(K_X+2H)=h^0(\xi)=h^0(\F)=\chi(\F)=3-\frac{1}{2}K_Bc_1(\F)+\frac{1}{2}c_1(\F)^2-c_2(\F)$$
that is (ii). Next, consider the exact sequences
\begin{equation}
\label{tan}
0 \to T_{X/B} \to T_X \to \pi^*T_B \to 0
\end{equation} 
and
$$0 \to \O_X \to \pi^* \F^*(\xi) \to T_{X/B} \to 0.$$
Since $T_X(2H)$ is Ulrich, we have that $\chi(T_X)=0$, hence, using Claim \ref{XeB}(ii) we get
\begin{equation}
\label{chi1}
\chi(T_B)=\chi(\pi^*T_B)=-\chi(T_{X/B})=-\chi(\pi^* \F^*(\xi))+1=-\chi(\F \otimes \F^*)+1.
\end{equation} 
On the other hand, by Riemann-Roch and Claim \ref{XeB}(v), $\chi(T_B)=2K_B^2-10$ and $\chi(\F \otimes \F^*)=9+2c_1(\F)^2-6c_2(\F)$. Replacing in \eqref{chi1} gives (iii).

Finally, to see (iv), we first compute $c_1(S^2(\F)) = 4c_1(\F)$ and $c_2(S^2(\F)) = 5c_1(\F)^2 + 5c_2(\F)$, so that
$$c_1(S^2(\F)(-M)) = -3K_B + c_1(\F), c_2(S^2(\F)(-M)) = \frac{15}{4}K_B^2 - \frac{5}{2}K_Bc_1(\F) - \frac{5}{4} c_1(\F)^2 + 5c_2(\F).$$
Now Riemann-Roch gives
\begin{equation}
\label{chi2}
\chi(S^2(\F)(-M)) = 6-K_Bc_1(\F)+\frac{9}{4}K_B^2 + \frac{7}{4} c_1(\F)^2 - 5c_2(\F).
\end{equation}
On the other hand, $\chi(S^2(\F)(-M))=\chi(2\xi-\pi^*M)=\chi(\O_X(H))=2+\chi(\O_S)-\frac{d}{6}$ by Claim \ref{XeB}(iii). Using \eqref{chi2} we get (v).
\end{proof}
We now conclude the proof of the lemma. 

Solving the five equations in Claim \ref{gro2} we get 
\begin{equation}
\label{sol}
K_B^2=-\frac{7}{48}d+7 \ \hbox{and} \ K_Bc_1(F)=-\frac{5}{48}d+9.
\end{equation} 
In particular $d \ge 48$. On the other hand, using Claim \ref{XeB}(i), we get
$$\mu(T_X)=\frac{-K_XH^3}{4}=\frac{1}{6}d$$
and, using \eqref{h}
$$\mu(\pi^*T_B) = \frac{c_1(\pi^*T_B)H^3}{2} = \frac{-\pi^*K_B\big[2\xi-\frac{1}{2}\pi^*(K_B+c_1(\F))\big]^3}{2}=-\frac{8\xi^3\pi^*K_B-6\xi^2\pi^*(K_B^2+K_Bc_1(\F))}{2}.$$
Now \eqref{gr2} and \eqref{sol} give
$$\mu(\pi^*T_B) =-K_Bc_1(F)+3K_B^2= -\frac{1}{3}d+12.$$
Since $T_X$ is semistable by Lemma \ref{posi}(v), we deduce by \eqref{tan} that 
$$\frac{1}{6}d \le -\frac{1}{3}d+12$$
that is $d \le 24$, a contradiction.
\end{proof}

\section{$T_X(1)$ Ulrich in any dimension}

We study the case $k=1$ in any dimension. We start analyzing the properties of the curve section $C$ and of the surface section $S$.

\begin{lemma}
\label{grado} 

Let $X \subseteq \P^N$ be a smooth irreducible variety of dimension $n \ge 3$. If $T_X(1)$ is an Ulrich vector bundle, then $d \ge 9$ except, possibly, when $d=8, g=5, n=4$ and $h^0(\O_C(1)) = 4$.
\end{lemma}
\begin{proof}
By Lemma \ref{posi}(i) we know that $g \ge 2$ and that
\begin{equation}
\label{con}
(n-1)d=(n+2)(g-1).
\end{equation}
By Lemma \ref{sezioni}(v) we have $d-g+1 = h^0(\O_C(1)) \ge 3$, hence $g \le d-2$. Also, if equality holds, then
$h^0(\O_C(1)) = 3$, so that $d-2 = g = \binom{d-1}{2}$, thus $d=3$ and $g=1$, a contradiction. Therefore $2 \le g \le d-3$, hence $d \ge 5$. But if $d \le 8$ the only possibility given by \eqref{con} is $d=8, g=5, n=4$ and $h^0(\O_C(1)) = 4$.
\end{proof}

\begin{lemma}
\label{varie-bis} 

Let $X \subseteq \P^N$ be a smooth irreducible variety of dimension $n \ge 3$. If $T_X(1)$ is an Ulrich vector bundle we have:
\begin{itemize} 
\item[(i)] $K_S H_{|S} = \frac{n-4}{n+2}d$.
\item[(ii)] $q(S)=p_g(S)=0$.
\item[(iii)] $K_S^2= - \frac{3(n-2)}{2(n+2)}d-\frac{n-12}{2}$.
\item[(iv)] $S$ is rational.
\end{itemize} 
\end{lemma}
\begin{proof}
(i) follows by Lemma \ref{posi}(ii), while (ii) follows by Lemma \ref{posi}(iii), Lemma \ref{q} and Lemma \ref{sezioni}(ii). Note now that the equation in Lemma \ref{posi}(vii) can be rewritten as
$$3(n-2)d+2(n+2)K_S^2+(n+2)(n-12)=0$$
giving (iii). Finally assume that $S$ is not ruled, so that $\kappa(S) \ge 0$. Then Lemma \ref{posi2} gives that $K_S+H_{|S}=(K_X+(n-1)H)_{|S}$ is nef, hence $K_S(K_S+H_{|S}) \ge 0$, that is $K_S^2 \ge - K_S H_{|S} = -\frac{n-4}{n+2}d$ by (i). Then (iii) gives 
$$- \frac{3(n-2)}{2(n+2)}d-\frac{n-12}{2} = K_S^2 \ge -\frac{n-4}{n+2}d$$
so that
$$(n+2)(d+n-12) \le 0.$$
Since $n \ge 3$ it follows that $d \le 9$, and using Lemma \ref{grado} we deduce that either $d=9, n=3$ or $d=8, g=5, n=4$ and $h^0(\O_C(1)) = 4$. In the first case we get a contraction by Lemma \ref{posi}(i), while in the second case $d+n-12=0$, hence $K_S^2=0$. As $C \subset \P^3$ we deduce that $S \subset \P^4$. But this contradicts the well-known formula for the invariants of a surface in $\P^4$. Therefore $S$ is ruled, hence rational by (ii) and (iv) is proved. 
\end{proof}

We are now ready to prove Theorem \ref{k=1}.

\renewcommand{\proofname}{Proof of Theorem \ref{k=1}}
\begin{proof}
If $n=1$ we know by Lemma \ref{posi}(i) that $T_X(1)$ is not an Ulrich vector bundle. If $n=2$ this is Theorem \ref{k=1sup}. Suppose next that $n \ge 3$. Note that $H^0(T_X)=0$ by Lemma \ref{coh}(iii), hence $X$ is neither $\P^n$ nor $Q_n$. Also $q(X)=0$ by Lemma \ref{posi}(iii). We have that $(X,H)$ is not:
\begin{itemize}
\item [(1)] A projective bundle over a smooth curve by Lemma \ref{notbundle}.
\item [(2)] A Del Pezzo manifold by Lemma \ref{posi}(i), since otherwise $g=1$.
\item [(3)] A hyperquadric fibration over a smooth curve (in the sense of \cite{i}), by Lemma \ref{noadj}(ii).
\item [(4)] A linear $\P^{n-2}$-bundle over a smooth surface, by Lemma \ref{noadj}(i).
\end{itemize}
Also observe that $X$ does not contain any exceptional divisor of type 1 by Lemma \ref{noexc}(i). Hence $(X,H)$ is isomorphic to its reduction $(X',H')$ (see \cite[(0.11)]{i}). It follows by \cite[Thm.~(1.7)]{i} that $K_X+(n-2)H$ is nef. Hence $S$ is minimal and rational by Lemma \ref{varie-bis}(iv), a contradiction since a minimal rational surface does not have nef canonical bundle. Thus the case $n \ge 3$ does not occur and the theorem is proved.
\end{proof} 
\renewcommand{\proofname}{Proof}

\section{$T_X(2)$ Ulrich in any dimension}

We prove Theorem \ref{k=2}. 

\renewcommand{\proofname}{Proof Theorem \ref{k=2}}
\begin{proof}
It follows by Lemma \ref{coh}(iii) that $H^0(T_X)=0$, hence $X$ is neither $\P^n$ nor $Q_n$. Note that $H^i(\O_X)=0$ for $i \ge 1$ by Lemma \ref{posi}(iii) and $K_X$ is not nef, since Lemma \ref{posi}(ii) gives that $K_X H^{n-1} = \frac{n(3-n)}{n+2}d<0$.

We divide the proof according to the value of $\tau(X,H)$ (see \eqref{tau}). We will also use the notions of first and second reduction of $(X,H)$, as defined in \cite[Defs.~7.3.3 and 7.5.7]{bs}.

\smallskip

\underline{Case A}: $\tau(X,H) \ge n-1$.

\smallskip

This case does not occur since Lemma \ref{gg}(iii) implies that $\tau(X,H) \le n - \frac{2n}{n+1} < n-1$.

\smallskip

\underline{Case B}: $n-2 \le \tau(X,H) < n-1$.

\smallskip

Then $K_X+(n-1)H$ is ample, hence the first reduction exists and is isomorphic to $(X,H)$. Therefore \cite[Thm.~7.3.4]{bs} implies that $\tau(X,H)=n-2$ and then \cite[Thm.~7.5.3]{bs} gives that $(X,H)$ is one of the following:
\begin{itemize}
\item [(B.1)] a Mukai variety,
\item [(B.2)] a Del Pezzo fibration over a smooth curve,
\item [(B.3)] a quadric fibration over a normal surface,
\item [(B.4)] a scroll over a normal threefold,
\item [(B.5)] $(X,H)$ contains an exceptional divisor of type 2, 3, or 4. 
\end{itemize}
Now, the case (B.1) is ruled out by Corollary \ref{sottoc}. Case (B.2) is excluded for $n=4$ by Lemma \ref{DP4}(i) and for $n \ge 5$ by Lemma \ref{noadj}(iii). Also the cases (B.3) and (B.4) are ruled out by Lemma \ref{noadj}(ii) and (i). Finally the case (B.5) is excluded by Lemma \ref{noexc}(ii) and (iii).

Thus also Case B does not occur.

\underline{Case C}: $\tau(X,H)< n-2$.

\smallskip

Then the first and second reductions exist and are both isomorphic to $(X,H)$, since $K_X+(n-2)H$ is ample.

We first claim that $K_{X_3}$ is not nef. In fact, assume that $K_{X_3}$ is nef. On the one hand, $\chi(\O_{X_3})=1$ by Lemma \ref{11}. On the other hand, $3c_2(X_3) - c_1(X_3)^2$ is pseff by \cite[Thm.~1.1]{m}, hence $3c_2(X_3)K_{X_3} \ge K_{X_3}^3 \ge 0$. But then Riemann-Roch gives $\chi(\O_{X_3}) = - \frac{1}{24}c_2(X_3)K_{X_3} \le 0$, a contradiction.

Hence $K_{X_3}$ is not nef and  \cite[Prop.~7.9.1]{bs} gives the following cases:
\begin{itemize}
\item [(C.1)] $n=5$ and $(X,K_X+3H)$ is a linear $\P^4$-bundle over a smooth curve.
\item [(C.2)] $n=4$ and $(X,K_X+2H)$ is a Del Pezzo.
\item [(C.3)] $n=4$ and $(X,K_X+2H)$ is a quadric fibration over a smooth curve.
\item [(C.4)] $n=4$ and $(X,K_X+2H)$ is a scroll over a normal surface.
\item [(C.5)] $n=4$ and $(X,H)$ contains an exceptional divisor of type 2.
\item [(C.6)] $n=4$ and $(X,K_X+2H)$ is a $(\P^3,\O_{\P^3}(2))$-fibration over a curve.
\end{itemize}
In case (C.1) we have a contradiction by Lemma \ref{notbundle}. In case (C.2) we have $4K_X+6H=0$, hence $4K_XH^3+6H^4=0$ and Lemma \ref{posi}(ii) gives the contradiction $d=0$. Cases (C.3) and (C.5) do not occur by Lemmas \ref{noqf4} and \ref{noexc}(ii). In Case (C.6) we observe that the fibration is obtained in \cite[(4.6.1)]{f2} by contracting an extremal ray, hence it is minimal (see Definition \ref{minimal}) and the image is a normal, hence smooth, curve. Thus this case is excluded by Lemma \ref{DP4}(ii).

Hence we are left with case (C.4). We have a surjective morphism $\pi: X \to B$ and denoting by $F$ a general fiber, we have $(F,(K_X+2H)_{|F}) \cong (\P^2,\O_{\P^2}(1))$. Now all fibers of $\pi$ are $2$-dimensional by \cite[Thm.~14.1.1]{bs}, hence we get by \cite[Prop.~3.2.1]{bs} that $B$ is a smooth surface and $(X,K_X+2H)$ is a linear $\P^2$-bundle over $B$. But this case is excluded by Lemma \ref{nosc4}. 

This concludes the proof of the theorem.
\end{proof}
\renewcommand{\proofname}{Proof}

\appendix
\section{Some numerical lemmas}

\begin{lemma}
\label{conto}
Let $(a;c_1,c_2,c_3,c_4) \in \Z^5$ be such that $c_1 \ge c_2 \ge c_3 \ge c_4 \ge 0, a \ge c_1+ c_2 + 3$ and
\begin{equation}
\label{1}
a^2-6a+4=c_1^2+c_2^2+c_3^2+c_4^2.
\end{equation}
Then  $(a;c_1,c_2,c_3,c_4) \in \{(6;2,0,0,0), (6;1,1,1,1), (7;3,1,1,0), (9;3,3,3,2)\}$.
\end{lemma}
\begin{proof}
We have 
\begin{equation}
\label{2}
a-3 \ge c_1+c_2.
\end{equation}
Now, \eqref{1} and \eqref{2} imply that $(a-3)^2-5=c_1^2+c_2^2+c_3^2+c_4^2 \ge (c_1+c_2)^2-5$, that is
\begin{equation}
\label{3}
c_3^2+c_4^2 \ge 2c_1c_2-5.
\end{equation}
But $2c_3^2 \ge c_3^2+c_4^2$ and $2c_1c_2 \ge 2c_2^2$. Consequently, we get 
\begin{equation}
\label{4}
5 \ge 2(c_2-c_3)(c_2+c_3).
\end{equation}
Thus, one of the following should happen:
\begin{itemize}
\item[($\alpha$)] $c_2=c_3$.
\item[($\beta$)] $c_2=c_3+1$.
\end{itemize}
First assume that case ($\beta$) holds.

Then \eqref{4} yields $2c_3 \le 1$ which gives $c_3=0, c_2=2$ and hence $c_4=o$. The \eqref{1} gives 
$$(a+c_1-3)(a-c_1-3)=6.$$ 
Thus we have one of the following possibilities
$$a+c_1=4, a-c_1=9, \ \hbox{or} \ a+c_1=5, a-c_1=6, \ \hbox{or} \ a+c_1=6, a-c_1=5, \ \hbox{or} \   \ a+c_1=9, a-c_1=4$$
but none of them have integer solutions. 

Assume now that case ($\alpha$) holds.

Set $c_2=c_3=c$. From \eqref{3}, we obtain $c^2+c_4^2 \ge 2c_1c-5$. Since $c \ge c_4$, we get 
\begin{equation}
5 \ge 2c(c_1-c).
\end{equation}
The above implies one of the following happens:
\begin{itemize}
\item[($\alpha1$)] $c_2=c_3=c_4=0$.
\item[($\alpha2$)] $c_1=c_2=c_3$.
\item[($\alpha3$)] $c_2=c_3=1$. This case has two sub-cases, namely $c_1=2,3$.
\item[($\alpha4$)] $c_2=c_3=1$, $c_1=3$.
\end{itemize}

Suppose we are in case ($\alpha1$). 

Then \eqref{1} gives $(a-3)^2-5=c_1^2$, so that $(a+c_1-3)(a-c_1-3)=5$. In this case, either $a+c_1=4, a-c_1=8$, giving the contradiction  $c_1=-2$, or $a+c_1=8, a-c_1=4$, giving $a=6, c_1=2$ and the solution $(6;2,0,0,0)$.

Suppose we are in case ($\alpha2$). 

Using \eqref{3} we conclude 
\begin{equation}
\label{6}
5 \ge (c-c_4)(c+c_4).
\end{equation}
As before, we obtain the following cases:
\begin{itemize}
\item[($\alpha21$)] $c=c_1=c_2=c_3=c_4$.
\item[($\alpha22$)] $c=c_4+1$.
\item[($\alpha23$)] $c=c_4+2$.
\end{itemize}

We first deal with ($\alpha21$). In this case, from \eqref{1}, we obtain 
$$(a+2c-3)(a-2c-3)=5.$$
Thus, we have either $a+2c=4, a-2c=8$, giving the contradiction $c=-1$, or $a+2c=8, a-2c=4$, giving the solution $(6;1,1,1,1)$.

We now deal with ($\alpha22$). From \eqref{6} we obtain $c+c_4-2 \le 5$, hence $c_4 \le 2$. Thus 
$(c, c_4) \in \{(3,2), (2,1), (1,0)\}$ and using \eqref{1} we see that it has no  integer solutions except in the first case, giving the solution $(9;3,3,3,2)$. 

We now deal with $(\alpha23)$. As before, in this case we have $c_4\le 0$. This implies $c=2, c_4=0$. But then \eqref{1} does not have any integer solution.
 
This concludes case ($\alpha2$). 

Suppose we are in case ($\alpha3$). 

We know that $(c_1,c_2,c_3,c_4) \in \{(1,1,1,1), (2,1,1,0), (3,1,1,1), (3,1,1,0)\}$. Using \eqref{1} we see that we have no integer solutions except in the last case, giving $(7;3,1,1,0)$.

Suppose we are in case $(\alpha4)$. 

Then $(c_1,c_2,c_3,c_4) \in \{(3,2,2,2),(3;2,2,1),(3;2,2,0)\}$ and \eqref{1} has no integer solutions.

This concludes case ($\alpha$) and the proof.
\end{proof}

\begin{lemma}
\label{632num}
Let ${\bf z}=(a;b_1,b_2,b_3,b_4,b_5,b_6) \in \Z^7$ with $b_1 \ge b_2 \ge b_3 \ge b_4 \ge b_5 \ge b_6$ satisfying the following
\begin{equation}
\label{0.1.1}
a^2-\sum_{i=1}^6b_i^2=10,\quad 3a-\sum_{i=1}^6b_i=6.
\end{equation}
Then ${\bf z} \in \left\{(4;1,1,1,1,1,1),(5;2,2,2,1,1,1),(6;3,2,2,2,2,1),(7;3,3,3,2,2,2),(8;3,3,3,3,3,3)\right\}$.
\end{lemma}
\begin{proof} 
We first use the Cauchy-Scwartz's inequality $(\sum_{i=1}^6b_i)^2 \le 6(\sum_{i=1}^6b_i^2)$ to obtain $(a^2-12a+32) \le 0$ whence $4 \le a \le 8$. We further observe that 
\begin{equation}
\sum_{i=1}^6(b_i^2-b_i)=a^2-3a-4.
\end{equation}
Also, $(b_i^2-b_i) \ge 0$ for all $i \ge 1$, and $b_1>0$ as $3a-6>0$ for $a \ge 4$.

\smallskip

\noindent\underline{Case 1: $a=4$.} We have $\sum_{i=1}^6(b_i^2-b_i)=0$ whence $|b_i| \le 1$ for all $i$. Since $\sum_{i=1}^6b_i=6$, we have $b_i=1$ for all $i$.

\smallskip

\noindent\underline{Case 2: $a=5$.} We have $\sum_{i=1}^6(b_i^2-b_i)=6$ whence $|b_i| \le 3$. Also, $\sum_{i=1}^6b_i^2=15$ and $\sum_{i=1}^6b_i=9$.

\noindent Subcase 2.1) $b_1=3$. In this case $\sum_{i=2}^6(b_i^2-b_i)=0$ whence $|b_i| \le 1$ for all $i \ge 2$. Consequently $\sum_{i=1}^6b_i \le 8$ which is a contradiction.

\noindent Subcase 2.2) $b_1 \le 2$. In this case we must have $b_1=b_2=b_3=2$. Consequently $\sum_{i=4}^6(b_i^2-b_i)=0$ whence $|b_i| \le 1$ for all $i \ge 4$ whence the only solution is ${\bf z}=(5; 2,2,2,1,1,1)$.

\smallskip

\noindent\underline{Case 3: $a=6$.} We have $\sum_{i=1}^6(b_i^2-b_i)=14$ whence $|b_i| \le 4$. Also $\sum_{i=1}^6b_i^2=26$ and $\sum_{i=1}^6b_i=12$.

\noindent Subcase 3.1) $b_1=4$. Then $\sum_{i=2}^6(b_i^2-b_i)=2$ whence $|b_i| \le 2$ for $i \ge 2$. Consequently $b_2=b_3=b_4=2$. But then $\sum_{i=1}^6b_i^2 \ge 28$ which is a contradiction.

\noindent Subcase 3.2) $b_1=3$.

3.2.1) $b_2=3$. In this case $\sum_{i=3}^6(b_i^2-b_i)=2$ whence $|b_i| \le 2$ for $i \ge 3$. Consequently, $b_3=b_4=2$. Thus $b_5+b_6=2$ and $b_5^2+b_6^2$ which is a contradiction.

3.2.2) $b_2 \le 2$. In this case we have the only solution ${\bf z}=(6; 3,2,2,2,2,1)$.

\noindent Subcase 3.3) $b_1 \le 2$. In this case $b_i=2$ for all $i$ whence $\sum_{i=1}^6b_i^2=24$ which is a contradiction.

\smallskip

\noindent\underline{Case 4: $a=7$.} We have $\sum_{i=1}^6(b_i^2-b_i)=24$ whence $|b_i| \le 5$. Also, $\sum_{i=1}^6b_i^2=39$ and $\sum_{i=1}^6b_i=15$.

\noindent Subcase 4.1) $b_1=5$. Then $\sum_{i=2}^6(b_i^2-b_i)=4$ whence $|b_i| \le 2$ for all $i \ge 2$. Consequently $b_i=2$ for all $i$, thus $\sum_{i=1}^6b_i^2=45$ which is a contradiction.

\noindent Subcase 4.2) $b_1=4$.

4.2.1) $b_2=4$. Then $\sum_{i=3}^6(b_i^2-b_i)=0$ whence $|b_i| \le 1$ for all $i \ge 3$. Consequently $\sum_{i=1}^6b_i \le 12$ which is a contradiction.

4.2.2) $b_2=3$.

\hspace{5pt} 4.2.2.1) $b_3=3$. Then $\sum_{i=4}^6(b_i^2-b_i)=0$ whence $|b_i| \le 1$ for $i \ge 4$. Consequently $\sum_{i=1}^6b_i \le 13$ which is a contradiction.

\hspace{5pt} 4.2.2.2) $b_3 \le 2$. Then $b_i=2$ for all $i \ge 3$ whence $\sum_{i=1}^6b_i^2 =41$ which is a contradiction.

4.2.3) $b_2 \le 2$. Then $\sum_{i=1}^6b_i \le 14$ which is a contradiction.

\noindent Subcase 4.3) $b_1=3$. Then $b_2=b_3=3$.

4.3.1) $b_4=3$. Then $\sum_{i=5}^6(b_i^2-b_i)=0$ whence $|b_i| \le 1$ for $i \ge 5$. Consequently $\sum_{i=1}^6b_i \le 14$ which is a contradiction.

4.3.2) $b_4 \le 2$. Then $b_4=b_5=b_6=2$. We get only one solution ${\bf z}=(7; 3,3,3,2,2,2)$.

\noindent Subcase 4.4) $b_1 \le 2$. Then $\sum_{i=1}^6b_i \le 12$ which is a contradiction.

\smallskip

\noindent\underline{Case 5: $a=8$.} We have $\sum_{i=1}^6(b_i^2-b_i)=36$ whence $|b_i| \le 6$. Also, $\sum_{i=1}^6b_i^2=54$ and $\sum_{i=1}^6b_i=18$.

\noindent Subcase 5.1) $b_1=6$. Then $\sum_{i=2}^6(b_i^2-b_i)=6$ whence $|b_i| \le 3$ for $i \ge 2$.

5.1.1) $b_2=3$. In this case $\sum_{i=3}^6(b_i^2-b_i)=0$ whence $|b_i| \le 1$ for $i \ge 3$. Consequently $\sum_{i=1}^6b_i \le 13$ which is a contradiction.

5.1.2) $b_2 \le 2$. In this case $\sum_{i=1}^6b_i \le 16$ which is a contradiction.

\noindent Subcase 5.2) $b_1=5$. Then $\sum_{i=2}^6(b_i^2-b_i)=16$ whence $|b_i| \le 4$.

5.2.1) $b_2=4$. Then $\sum_{i=3}^6(b_i^2-b_i)=4$ whence $|b_i| \le 2$ for $i \ge 3$. Thus $\sum_{i=1}^6b_i \le 17$ which is a contradiction.

5.2.2) $b_2=3$ which implies $b_3=3$. Then $\sum_{i=4}^6(b_i^2-b_i)=4$ whence $|b_i| \le 2$ for $i \ge 4$. Consequently $\sum_{i=1}^6b_i \le 17$ which is a contradiction.

5.2.3) $b_2 \le 2$. Then $\sum_{i=1}^6b_i \le 15$ which is a contradiction.

\noindent Subcase 5.3) $b_1=4$. 

5.3.1) $b_2=4$.

\hspace{5pt} 5.3.1.1) $b_3=4$. Then $\sum_{i=4}^6(b_i^2-b_i)=0$ whence $|b_i| \le 1$ for $i \ge 4$. Consequently $\sum_{i=1}^6b_i \le 15$ which is a contradiction.

\hspace{5pt} 5.3.1.2) $b_3=3$ which implies $b_4=3$. Thus $\sum_{i=5}^6(b_i^2-b_i)=0$ whence $|b_i| \le 1$ for $i \ge 5$. Consequently $\sum_{i=1}^6b_i \le 16$ which is a contradiction.

\hspace{5pt} 5.3.1.3) $b_3 \le 2$. Then $\sum_{i=1}^6b_i \le 16$ which is a contradiction.

5.3.2) $b_2=3$. Then $b_3=b_4=b_5=3$ and $b_6=2$. Consequently $\sum_{i=1}^6b_i^2=56$ which is a contradiction.

5.3.3) $b_2 \le 2$. Then $\sum_{i=1}^6b_i \le 14$ which is a contradiction.

\noindent Subcase 5.4) $b_1 \le 3$. In this case we have the only solution ${\bf z}=(8;3,3,3,3,3,3)$.
\end{proof}


\begin{thebibliography}{CMRPL}

\bibitem[A]{a} F.~Ambro.
\textit{Ladders on Fano varieties}. 
Algebraic geometry, 9. J. Math. Sci. (New York)  \textbf{94} (1999), no.\ 1, 1126-1135.

\bibitem[Be1]{b1} A.~Beauville.
\textit{An introduction to Ulrich bundles}. 
Eur. J. Math. \textbf{4} (2018), no.~1, 26-36.

\bibitem[Be2]{b2} A.~Beauville.
\textit{Complex algebraic surfaces}. 
London Mathematical Society Lecture Note Series, \textbf{68}. Cambridge University Press, Cambridge, 1983. iv+132 pp. 

\bibitem[BC]{bc} I.~Bauer, F.~Catanese.
\textit{On rigid compact complex surfaces and manifolds}. 
Adv. Math. \textbf{333} (2018), 620-669. 

\bibitem[BMQ]{bmq} F.~Bogomolov, M.~McQuillan.
\textit{Rational curves on foliated varieties}. 
Foliation theory in algebraic geometry, 21-51, Simons Symp., Springer, Cham, 2016. 

\bibitem[BMPT]{bmpt} V.~Benedetti, P.~Montero, Y.~Prieto Monta\~{n}ez, S.~Troncoso.
\textit{Projective manifolds whose tangent bundle is Ulrich}. 
J. Algebra \textbf{630} (2023), 248-273.

\bibitem[Bo]{bo} F.~A.~Bogomolov.
\textit{Unstable vector bundles and curves on surfaces}. 
Proceedings of the International Congress of Mathematicians (Helsinki, 1978), pp. 517-524, Acad. Sci. Fennica, Helsinki, 1980. 

\bibitem[BS]{bs} M.~C.~Beltrametti, A.~J.~Sommese.
\textit{The adjunction theory of complex projective varieties}. 
De Gruyter Expositions in Mathematics, \textbf{16}. Walter de Gruyter \& Co., Berlin, 1995.

\bibitem[C1]{c} G.~Casnati.
\textit{Special Ulrich bundles on non-special surfaces with $p_g=q=0$}. 
Internat. J. Math. \textbf{28} (2017), no. 8, 1750061, 18 pp.

\bibitem[C2]{c2} G.~Casnati.
\textit{Tangent, cotangent, normal and conormal bundles are almost never instanton bundles}.
Preprint 2023, arXiv:2303.04064.

\bibitem[CH]{ch} M.~Casanellas, R.~Hartshorne.
\textit{Stable Ulrich bundles}. With an appendix by F.~Geiss, F.-O.~Schreyer.
Internat. J. Math. \textbf{23} (2012), no. 8, 1250083, 50 pp.

\bibitem[CMRPL]{cmp} L.~Costa, R.~M.~Mir\'o-Roig, J.~Pons-Llopis.
\textit{Ulrich bundles}.
De Gruyter Studies in Mathematics, \textbf{77}, De Gruyter 2021. 

\bibitem[CP]{cp} F.~Campana, M.~P$\breve{\rm a}$un.
\textit{Foliations with positive slopes and birational stability of orbifold cotangent bundles}. 
Publ. Math. Inst. Hautes \'Etudes Sci. \textbf{129} (2019), 1-49. 

\bibitem[DR]{d} S.~Di Rocco.
\textit{$k$-very ample line bundles on del Pezzo surfaces}. 
Math. Nachr. \textbf{179} (1996), 47-56.

\bibitem[ES]{es} D.~Eisenbud, F.-O.~Schreyer.
\textit{Resultants and Chow forms via exterior syzygies}. 
J. Amer. Math. Soc. \textbf{16} (2003), no.~3, 537-579.

\bibitem[F1]{f1}
T. Fujita.
\textit{On del Pezzo fibrations over curves}. 
Osaka J. Math. \textbf{27} (1990), no.~2, 229-245.

\bibitem[F2]{f2}
T. Fujita.
\textit{On Kodaira energy and adjoint reduction of polarized manifolds}. 
Manuscripta Math. \textbf{76} (1992), no. 1, 59-84.

\bibitem[FL1]{fl1} M.~Fulger, B.~Lehmann. 
\textit{Morphisms and faces of pseudo-effective cones}. 
Proc. Lond. Math. Soc. \textbf{112} (2016), no. 4, 651-676.

\bibitem[FL2]{fl2} M.~Fulger, B.~Lehmann. 
\textit{Positive cones of dual cycle classes}. 
Algebr. Geom. \textbf{4} (2017), no. 1, 1-28.

\bibitem[GL]{gl}
L.~Ghidelli, J.~Lacini.
\textit{Logarithmic bounds on Fujita's conjecture}. 
Preprint 2021, arXiv:2107.11705.

\bibitem[Ha1]{ha1} R.~Hartshorne. 
\textit{Algebraic geometry}. 
Graduate Texts in Mathematics \textbf{52}. Springer-Verlag, New York-Heidelberg, 1977.

\bibitem[Ha2]{ha2} R.~Hartshorne. 
\textit{On the classification of algebraic space curves}.
Vector bundles and differential equations (Proc. Conf., Nice, 1979), pp. 83-112, Progr. Math.,  \textbf{ 7}, Birkhäuser, Boston, Mass., 1980

\bibitem[Ho]{ho} A.~H\"oring.
\textit{On a conjecture of Beltrametti and Sommese}.
J. Algebraic Geom.  \textbf{21} (2012), no.\ 4, 721-751.

\bibitem[I]{i} P.~Ionescu.
\textit{Generalized adjunction and applications}.
Math. Proc. Cambridge Philos. Soc. \textbf{99} (1986), no.\ 3, 457-472. 

\bibitem[IP]{ip}
V. A. Iskovskikh, Yu. Prokhorov.
\textit{Fano varieties}. 
In: Algebraic geometry, V, 1-247, Encyclopaedia Math. Sci. \textbf{47}, Springer-Verlag, Berlin, 1999.

\bibitem[K]{k} Y.~Kawamata.
\textit{On effective non-vanishing and base-point-freeness}.
Kodaira's issue. Asian J. Math. \textbf{4} (2000), no.\ 1, 173-181.

\bibitem[Lan]{la} A.~Lanteri. 
\textit{Hilbert curves of quadric fibrations}.
Internat. J. Math. \textbf{29} (2018), no.\ 10, 1850067, 20 pp. 

\bibitem[Lop]{lo} A.~F.~Lopez. 
\textit{On varieties with Ulrich twisted normal bundles}.
Preprint 2022, arXiv:2205:06602. To appear on {\it Perspectives on four decades: Algebraic Geometry 1980-2020. In memory of Alberto Collino. Trends in Mathematics, Birkh\"auser}.

\bibitem[Laz]{laz1} R.~Lazarsfeld. \textit{Positivity in algebraic geometry, I}.
Ergebnisse der Mathematik und ihrer Grenzgebiete, 3. Folge  \textbf{48}, Springer-Verlag, Berlin, 2004.

\bibitem[LS]{ls} A.~F.~Lopez, J.~C.~Sierra. 
\textit{A geometrical view of Ulrich vector bundles}.
Int. Math. Res. Not. IMRN(2023), no.\ 11, 9754-9776.

\bibitem[M]{m} Y.~Miyaoka. 
\textit{The Chern classes and Kodaira dimension of a minimal variety}.
In: Algebraic geometry, Sendai, 1985, 449-476, Adv. Stud. Pure Math., \textbf{10}, North-Holland, Amsterdam, 1987.

\bibitem[MS]{ms} S.~Mori, H.~Sumihiro.
\textit{On Hartshorne's conjecture}. 
J. Math. Kyoto Univ. \textbf{18} (1978), no. 3, 523-533.

\bibitem[T]{t} B.~Totaro.
\textit{Bott vanishing for algebraic surfaces}. 
Trans. Amer. Math. Soc. \textbf{373} (2020), no. 5, 3609-3626. 

\bibitem[W]{w2} J.~M.~Wahl.
\textit{A cohomological characterization of $\P^n$}. 
Invent. Math. \textbf{72} (1983), no. 2, 315-322.

\end{thebibliography}
\end{document}